\documentclass[a4paper]{amsart}

\usepackage[a4paper,hmargin=4cm,vmargin=4cm]{geometry}
\usepackage{amsfonts,amssymb,amscd,amstext}
\usepackage{graphicx}
\usepackage[dvips]{epsfig}

\usepackage[colorlinks=true,linkcolor=blue,citecolor=red]{hyperref}

\renewcommand{\leq}{\leqslant}
\renewcommand{\geq}{\geqslant}
\newcommand{\ptl}{\partial}
\newcommand{\rr}{{\mathbb{R}}}
\newcommand{\la}{\lambda}

\newcommand{\nn}{{\mathbb{N}}}

\newcommand{\sub}{\subset}
\newcommand{\subeq}{\subseteq}
\newcommand{\escpr}[1]{\big<#1\big>}

\newcommand{\sg}{\sigma}
\newcommand{\Om}{\Omega}
\newcommand{\eps}{\varepsilon}
\newcommand{\var}{\varphi}

\DeclareMathOperator{\divv}{div}

\newtheorem{theorem}{Theorem}[section]

\newtheorem{corollary}[theorem]{Corollary}

\renewcommand{\thetheoremName}

\theoremstyle{definition}
\newtheorem{definition}[theorem]{Definition}
\newtheorem{remark}[theorem]{Remark}
\newtheorem{example}[theorem]{Example}

\theoremstyle{remark}

\numberwithin{equation}{section}

\newcommand{\Hess}{\operatorname{Hess}}

\newcommand{\Vol}{\operatorname{Vol}}

\newcommand{\C}{\operatorname{Cap}}

\newcommand{\kam}{\mathbb{K}^{m}(b)}

\numberwithin{equation}{section}

\newcommand{\ovh}{\overline{H}_P}

\begin{document}
\title[Comparison results in weighted manifolds]
{Intrinsic and extrinsic comparison results for isoperimetric quotients and capacities \\ in weighted manifolds}

\author[A. Hurtado]{A. Hurtado$^{\natural}$}
\address{Departamento de Geometr\'{\i}a y Topolog\'{\i}a and Excellence Research Unit ``Modeling Nature'' (MNat), Universidad de Granada, E-18071,
Spain.}
 \email{ahurtado@ugr.es}
 
\author[V. Palmer]{V. Palmer*}
\address{Departament de Matem\`{a}tiques, Universitat Jaume I, Castell\'o,
Spain.} \email{palmer@mat.uji.es}

\author[C. Rosales]{C. Rosales$^{\#}$}
\address{Departamento de Geometr\'{\i}a y Topolog\'{\i}a and Excellence Research Unit ``Modeling Nature'' (MNat) Universidad de Granada, E-18071,
Spain.} 
\email{crosales@ugr.es}

\date{\today}

\thanks{* Supported by MINECO grant MTM2017-84851-C2-2-P and UJI grant UJI-B2018-35. \\
\indent $^{\natural}$ $^{\#}$ Supported by MINECO grant MTM2017-84851-C2-1-P and Junta de Andaluc\'{\i}a grant FQM325.}



\begin{abstract}
Let $(M,g)$ be a complete non-compact Riemannian manifold together with a function $e^h$, which weights the Hausdorff measures associated to the Riemannian metric. In this work we assume lower or upper radial bounds on some weighted or unweighted curvatures of $M$ to deduce comparisons for the weighted isoperimetric quotient and the weighted capacity of metric balls in $M$ centered at a point $o\in M$. As a consequence, we obtain parabolicity and hyperbolicity criteria for weighted manifolds generalizing previous ones. A basic tool in our study is the analysis of the weighted Laplacian of the distance function from $o$. The technique extends to non-compact submanifolds properly immersed in $M$ under certain control on their weighted mean curvature.
\end{abstract}

\maketitle
\thispagestyle{empty}

\section{Introduction}
\label{secIntro}

A \emph{weighted manifold} (also known as a \emph{manifold with density} or \emph{smooth metric measure space}) is a triple $(M,g,e^h)$, where $(M,g)$ is a Riemannian manifold and $e^h$ a smooth function used to weight the Hausdorff measures associated to the Riemannian distance. In these manifolds, by combining the Riemannian structure with the derivatives of $h$, it is possible to introduce weighted curvatures and differential operators, see Morgan's book~\cite[Ch.~18]{gmt} and Section~\ref{secPrelim} for precise definitions. Hence, the framework of weighted manifolds provides an extension of Riemannian geometry where many classical questions are being analyzed in recent years. In particular, comparison geometry and topological obstructions for the different weighted curvatures have been considered by many authors, see \cite{qian2,qian,lott,MR2243959,morganmyers,WW,MR2567278,MRS,MR3197658,W,MR3641482,MR3897040} but this list is far from exhaustive.

Our aim in this paper is to establish comparison results for the weighted isoperimetric quotients and capacities of intrinsic and extrinsic balls in a complete non-compact weighted manifold. These will be derived from lower or upper radial (maybe non-constant) bounds on some of the curvatures of the manifold, in such a way that they become equalities in the corresponding comparison model. The \emph{weighted model spaces} are defined in Section~\ref{secModel} as rotationally symmetric manifolds with a pole together with a weight which is radial, i.e., it only depends on the Riemannian distance from the pole. Our capacity inequalities will allow to deduce parabolicity (resp. hyperbolicity) criteria for weighted manifolds and their proper submanifolds from the parabolicity (resp. hyperbolicity) of the associated comparison models. 

The starting point for our results is the analysis of $\Delta^h r$, where $\Delta^h$ is the \emph{weighted Laplacian} defined in \eqref{eq:flaplacian} and $r$ denotes the distance function in $M$ from a fixed point $o\in M$. This is accomplished in Section~\ref{distanceanalysis} by means of a standard method in Riemannian geometry \cite[Ch.~2]{GreW}, \cite{Zhu}. From inequality \eqref{eq:seque}, which relates the Hessian $\text{Hess}\,r$ and the sectional curvatures in $(M,g)$, we can control $\Delta^hr$ by assuming radial lower bounds on two kinds of weighted curvatures: the  \emph{Bakry-\'Emery Ricci curvatures} $\text{Ric}^h_q$ and the \emph{weighted sectional curvatures} $\text{Sec}^h_q$ with $q\in(0,\infty]$, see Section~\ref{subsec:wc} for precise definitions and references. Estimates for $\Delta^h r$ involving lower bounds on $\text{Ric}^h_q$ have appeared in many of the aforementioned works by following the same technique or by using a Bochner-Weitzenb\"ock formula in weighted manifolds. A similar approach was followed in \cite{W,MR3641482} to extend classical Riemannian statements to the case $\text{Sec}^h_q>0$. Our more general comparisons for $\text{Hess}\,r$ in Theorem~\ref{th:hessianineqr} depending on lower radial bounds on the sectional curvatures $\text{Sec}^h_q$ seem to be new.

Having the Laplacian inequalities for the distance in hand, we are ready to establish our main comparisons in a unified way. For more clarity we have divided the exposition into two sections where we treat separately the intrinsic case and the extrinsic one. 

Section~\ref{intcompan} is devoted to the intrinsic setting and contains two type of results. The first ones are estimates for isoperimetric quotients of a complete non-compact weighted manifold $(M,g,e^h)$. Given a point $o\in M$, the \emph{weighted isoperimetric quotient} $q_o(R)$ measures the ratio between the weighted volume of the open metric ball $B_R$ of radius $R$ centered at $o$ and the weighted area of $\ptl B_R$. In Theorems~\ref{th:isoper_infty}, \ref{volumeRiem1} and \ref{th:isoper} we compare $q_o(R)$ with the weighted isoperimetric quotient at the pole of a weighted model space, which is determined from eventual bounds on $\escpr{\nabla h,\nabla r}$ and radial lower bounds on weighted or unweighted Ricci curvatures of $(M,g,e^h)$. In Theorem~\ref{volumeRiem2} we show the opposite comparison under an upper bound on the Riemannian sectional curvatures. For the proofs we adapt to the weighted context the arguments employed in \cite{Pa,MaP1,MP5} for submanifolds in Riemannian manifolds with a pole. The main idea is to use the previous analysis of $\Delta^h r$ to compare the mean exit time function in $B_R$ with the function defined by transplanting to $B_R$, via the distance function $r$, the mean exit time function for the ball of the same radius in the weighted model space. As a consequence of our estimates for $q_o(R)$ we can compare separately the weighted volumes and areas of metric balls and spheres in $(M,g,e^h)$ with the ones in the corresponding model. Similar inequalities for weighted volumes and quotients of weighted volumes, but not for weighted isoperimetric quotients, were given in \cite{qian2,lott,morganmyers,WW,PRRS,MR3197658} under a lower bound on $\text{Ric}^h_\infty$, and in \cite{qian,MR2243959,MRS} under a lower bound on $\text{Ric}^h_q$ with $q\in(0,\infty)$. 

In Section~\ref{intcompan} we also establish comparisons for the capacities of metric balls, and deduce from them parabolicity and hyperbolicity criteria. To explain this in more detail we need to introduce some notation and definitions.

Following classical terminology in potential theory, a weighted manifold $(M,g,e^h)$ is \emph{weighted parabolic} or \emph{$h$-parabolic} if every function $u\in C^\infty(M)$ bounded from above and satisfying $\Delta^hu\geq 0$ must be constant. Otherwise, $(M,g,e^h)$ is \emph{weighted hyperbolic}. As in the unweighted setting, the $h$-parabolicity is characterized in terms of the \emph{$h$-capacities} defined in Section~\ref{sec:wcap}. More precisely, by Theorem~\ref{theorGrig}, the $h$-parabolicity is equivalent to the existence of a precompact open set $D\subseteq M$, and an exhaustion $\{\Om_i\}_{i\in\nn}$ of $M$ by smooth precompact open sets, such that $\lim_{i\to\infty}\C^h(D,\Om_i)=0$. When $\ptl D$ is smooth the value of $\text{Cap}^h(D,\Om_i)$ can be computed from equality \eqref{eq:capint} involving the \emph{$h$-capacity potential}, which is the solution to the weighted Laplace equation with Dirichlet boundary condition in \eqref{laplace}. This characterization leads to parabolicity and hyperbolicity criteria when we are able to bound the capacities $\C^h(D,\Om_i)$ for a suitable family of capacitors $(D,\Om_i)$ in terms of the geometry of the underlying weighted manifold. Following this idea with the capacitors $(B_\rho,B_R)$, the $h$-parabolicity of a weighted manifold comes from suitable growth properties of the weighted volume and boundary area of the balls $B_t$, see for instance \cite{grigoryan2,gri-masa}. In the case of a weighted model $(M_w,g_w,e^{f(r)})$, an explicit computation of $\C^f(B^w_\rho,B^w_R)$, where $B^w_t$ denotes a metric ball centered at the pole, shows that the $f$-parabolicity is equivalent to that $\int_\rho^\infty\Vol_f^{-1}(\ptl B^w_t)\,dt=\infty$ for some $\rho>0$, see \cite{grigoryan2,HPR1}. This is a weighted extension of a classical parabolicity criterion of Ahlfors for rotationally symmetric Riemannian manifolds \cite{Ahl,grigoryan}.

In Theorems~\ref{comp:inftycapacity}, \ref{parabolicity-h''} and \ref{comp:qcapacity} we prove, for almost every radii $\rho<R$, a comparison for the ratio between the capacity $\C^h(B_\rho,B_R)$ and the weighted area of $\ptl B_\rho$ with respect to the same ratio in a model space with weight depending on an eventual bound on $\escpr{\nabla h,\nabla r}$, and with curvature determined from a lower bound on a weighted or unweighted Ricci curvature in $(M,g,e^h)$. In Theorem \ref{hyperbolicity-h''} we deduce the opposite comparison from an upper bound on the Riemannian sectional curvature. The proofs employ the analysis of $\Delta^h r$ to compare the $h$-capacity potential of $(B_\rho,B_R)$ with the function obtained by transplanting to the annulus $B_R-\overline{B}_\rho$, via the distance function $r$, the capacity potential of the capacitor $(B^w_\rho,B^w_R)$ in the weighted model. By passing to the limit when $R\to\infty$ the resulting capacity estimates imply the $h$-parabolicity (resp. $h$-hyperbolicity) of $(M,g,e^h)$ from the parabolicity (resp. hyperbolicity) of the corresponding model. In particular, we generalize to weighted manifolds a parabolicity (resp. hyperbolicity) criterion of Ichihara~\cite{I1} for Riemannian manifolds with Ricci curvature bounded from below (resp. sectional curvature bounded from above). Also, by combining these capacity estimates with our previous inequalities for the weighted area of metric spheres, we achieve a direct comparison between $\C^h(B_\rho)$ and $\C^f(B^w_\rho)$.

The proof technique in these results can also be employed to study the parabolicity of a Riemannian manifold $(M,g)$. It is easy to observe that the $h$-parabolicity of $(M,g,e^h)$ does not imply the parabolicity of $(M,g)$. For instance, Euclidean space $\rr^m$ is hyperbolic for $m\geq 3$, whereas it is parabolic with respect to the Gaussian weight. In this sense, it is interesting to provide sufficient conditions on a weight $e^h$ ensuring the Riemannian parabolicity of $(M,g)$. In Theorem~\ref{comp:capacity} we give a criterion in this line which involves a lower bound on some Bakry-\'Emery Ricci curvature $\text{Ric}^h_q$ with $q\in(0,\infty)$.

In Section~\ref{extrinsic} we gather our comparisons in the extrinsic setting. This arises when a submanifold $P$ immersed in the manifold $M$ is endowed with the structure of weighted manifold inherited from the Riemannian metric and weight in $(M,g,e^h)$. More precisely, we suppose that $M$ has a pole and $P$ is a non-compact submanifold with empty boundary properly immersed in $M$. By following similar arguments as in the intrinsic case we deduce, from the weighted geometry of $M$ and the extrinsic weighted geometry of $P$, sharp estimates for the weighted isoperimetric quotients and capacities of extrinsic balls. By an \emph{extrinsic ball} in $P$ we mean any connected component in the intersection of $P$ with the metric ball $B_R$ in $M$ centered at the pole.

Our results here rely on the analysis of the \emph{extrinsic distance} $r_{|P}$ developed in Section~\ref{extLapcomp}. In this situation, the Hessian in $P$ depends on the Hessian in $M$ and the second fundamental form of $P$. As in the Riemannian setting, see for instance \cite{Pa2}, we may assume suitable bounds on weighted or unweighted sectional curvatures in $(M,g,e^h)$ to control the weighted Laplacian $\Delta^h_P$ of radial functions in $P$ by means of our previous analysis on $\text{Hess}\,r$, and from radial bounds on the term $\escpr{\nabla h+\ovh^h,\nabla r}$, which involves the radial derivatives of $h$ and the \emph{weighted mean curvature vector} $\ovh^h$ defined in Section~\ref{submanifolds}. In the particular case of weighted rotationally symmetric manifolds with a pole, the sectional curvature is a radial function from the pole, and the Laplacian $\Delta^h_P$ can be explicitly computed as we did in \cite{HPR1}.

In Theorem~\ref{th:simpson} we show how a radial bound on the sectional curvature $\text{Sec}$ in $(M,g)$ implies an inequality for the weighted isoperimetric quotient of extrinsic balls in $P$. Under a lower bound on $\text{Sec}$ we generalize a result of Markvorsen and the second author in the Riemannian setting~\cite{MP5}. Under an upper bound on $\text{Sec}$, our inequality provides a genuine comparison with respect to the isoperimetric quotient of a weighted model space, thus extending a result of the second author~\cite{Pa} for minimal submanifolds of Cartan-Hadamard manifolds. 

In Theorems~\ref{parabolicity-sub-h''} and \ref{parabolicity-sub-gen-q}, we consider capacitors associated to concentric extrinsic balls and estimate their capacities to obtain, by means of Theorem~\ref{theorGrig}, parabolicity and hyperbolicity criteria for submanifolds under certain control on the weighted or unweighted sectional curvatures of the ambient manifold, and the weighted mean curvature vector of the submanifold. As in the intrinsic setting, we can conclude the $h$-parabolicity (resp. $h$-hyperbolicity) of $P$ from the parabolicity (resp. hyperbolicity) of the associated comparison model. Our Theorem~\ref{parabolicity-sub-h''} extends to arbitrary weighted manifolds previous criteria of the authors \cite{HPR1} for rotationally symmetric manifolds with weights. In the unweighted case $h=0$ we recover previous statements by Esteve and the second author~\cite{esteve-palmer}, and by Markvorsen and the second author~\cite{MP3}. In Corollary~\ref{hiperbolicity-MP} we show a weighted version of the hyperbolicity result of Markvorsen and the second author~\cite{MP2} for minimal submanifolds of a Cartan-Hadamard manifold.

To finish this introduction we should mention that, due to the relation between the weighted and the unweighted curvatures, all the results of the paper depending on a certain bound on the Ricci curvature $\text{Ric}$ or the sectional curvature $\text{Sec}$ are valid when we assume the same bound on $\text{Ric}^h_\infty$ or $\text{Sec}^h_\infty$ together with an additional concavity or convexity condition on $h$. 

The paper is organized into five sections. Section~\ref{secPrelim} contains some preliminary material about weighted manifolds, potential theory and weighted model spaces. In Section~\ref{distanceanalysis} we derive estimates for the weighted Laplacian of the distance function in terms of lower bounds on the weighted curvatures of the manifold. In Section~\ref{intcompan} we prove our comparison results and our parabolicity / hyperbolicity criteria for weighted manifolds. Section~\ref{extrinsic} includes the analysis of the weighted Laplacian and the comparisons for submanifolds with controlled weighted mean curvature vector.

\section{Preliminaries}
\label{secPrelim}

In this section we recall some notions and results that will be
instrumental in the sequel.

\subsection{Weighted capacities and parabolicity} 
\label{sec:wcap}\

In a complete Riemannian manifold $(M^m,g)$ with $m\geq 2$ and $\ptl M=\emptyset$ we consider a {\em weight} or {\em density}, i.e., a smooth positive function $e^h$ on $M$ which is used to weight the Hausdorff measures associated to the Riemannian metric. In particular, for any Borel set $E\subeq M$, and any $C^1$ hypersurface $P\sub M$, the \emph{weighted volume} of $E$ and the \emph{weighted area} of $P$ are given by
\begin{equation*}
\Vol_h(E):=\int_E dv_h=\int_E e^h\,dv,\quad \Vol_h(P):=\int_P da_h=\int_P e^h\,da,
\end{equation*}
where $dv$ and $da$ denote the Riemannian elements of volume and area, respectively.

In weighted manifolds there are generalizations not only of volume and area, but also of some differential operators of Riemannian manifolds. The \emph{weighted divergence} of a smooth vector field $X$ on $M$ is the function
\[
\divv^hX:=\divv X+\escpr{\nabla h,X},
\]
where $\divv$ is the Riemannian divergence, $\nabla$ is the Riemannian gradient and $\escpr{\cdot\,,\cdot}$ denotes the Riemannian metric in $M$. Following \cite[Sect.~3.6]{gri-book} we define the \emph{weighted Laplacian} or \emph{$h$-Laplacian} of a function $u\in C^2(M)$ by
\begin{equation}
\label{eq:flaplacian}
\Delta^hu:=\divv^h\nabla u=\Delta u+\escpr{\nabla h,\nabla u},
\end{equation}
where $\Delta$ is the Laplacian in $(M,g)$. The $h$-Laplacian is a second order linear operator, which is self-adjoint with respect to $dv_h$ since
\[
\int_Mu\,\Delta^hw\,dv_h=\int_Mw\,\Delta^hu\,dv_h,
\]
for any two functions $u,w\in C^2_0(M)$.

Given a domain (connected open set) $\Om$ in $M$, a function $u\in C^2(\Om)$ is \emph{$h$-harmonic} (resp. \emph{$h$-subharmonic}) if $\Delta^h u=0$ (resp. $\Delta^h u\geq 0$) on $\Om$. As in the unweighted setting there is a strong maximum principle and a Hopf boundary point lemma for $h$-subharmonic functions. We gather both results in the next statement, see \cite[Sect.~8.3]{gri-book} and \cite[Sect.~3.2]{gilbarg-trudinger}.

\begin{theorem}
\label{th:mp}
Let $\Om$ be a smooth domain of a weighted manifold $(M^m,g,e^h)$. Consider an $h$-subharmonic function $u\in C^2(\Om)\cap C(\overline{\Om})$. Then, we have:
\begin{itemize}
\item[(i)] if $u$ achieves its maximum in $\Om$ then $u$ is constant,
\item[(ii)] if there is $p_0\in\ptl\Om$ such that $u(p)<u(p_0)$ for any $p\in\Om$ then $\frac{\ptl u}{\ptl\nu}(p_0)>0$, where $\nu$ denotes the outer unit normal along $\ptl\Om$.
\end{itemize}
\end{theorem}

From the maximum principle it is clear that any $h$-subharmonic function on a compact manifold $M$ must be constant. In general, a weighted manifold is \emph{weighted parabolic} or \emph{$h$-parabolic} if any $h$-subharmonic function which is bounded from above must be constant. Otherwise we say that $M$ is \emph{weighted hyperbolic} or \emph{$h$-hyperbolic}. 

Next, we will recall how the $h$-parabolicity of manifolds can be characterized by means of weighted capacities. For more details about the definitions and results below we refer to \cite{grigoryan,grigoryan-escape,gri-masa}.

Let $\Om\subeq M$ be an open set and $K\sub\Om$ a compact set. The \emph{weighted Newtonian capacity} or {\em $h$-capacity} of the \emph{capacitor} $(K,\Om)$ is defined by
\begin{equation*}
\text{Cap}^{h}(K,\Om):=\inf\left\{\int_\Om|\nabla\phi|^2\,dv_h\,;\,\phi\in
H^1_0(\Om,dv_h) \text{ with }0\leq\phi\leq 1\text{ and
}\phi=1\text{ on }K\right\}.
\end{equation*}
Here $H^1_0(\Om,dv_h)$ denotes the closure of $C^\infty_0(\Om)$
with respect to the norm $\|u\|:=(\int_\Om
u^2\,dv_h+\int_\Om |\nabla u|^2\,dv_h)^{1/2}$ in the weighted Sobolev space $H^1(\Om,dv_h)$ of functions $u\in L^2(\Om,dv_h)$ with distributional gradient satisfying $|\nabla u|\in L^2(\Om,dv_h)$. By a
standard approximation argument it follows that
\begin{equation*}
\label{eq:capacity2}
\text{Cap}^{h}(K,\Om)=\inf\left\{\int_\Om|\nabla\phi|^2\,dv_h\,;\,\phi\in
C^\infty_0(\Om) \text{ with }0\leq\phi\leq 1\text{ and
}\phi=1\text{ on }K\right\}.
\end{equation*}
For a precompact open set $D$ with $\overline{D}\sub\Om$ we denote $\text{Cap}^h(D,\Om):=\text{Cap}^h(\overline{D},\Om)$. For $\Omega=M$ we write $\C^h(D):=\C^h(D,M)$ and we call it the \emph{$h$-capacity of $D$ at infinity}. It can be proved that equality
\begin{equation}
\label{eq:sensible}
\C^h(D)=\lim_{i \to \infty}\C^h(D, \Omega_i)
\end{equation}
holds for any exhaustion $\{\Omega_i\}_{i=1}^{\infty}$ of $M$ by precompact open sets. This means that $\cup_{i=1}^\infty\Om_i=M$ and $\overline{\Om}_i\sub\Om_{i+1}$ for any $i\in\nn$.

When $\Omega\subset M$ is a precompact open set and $K\subset\Om$ has smooth boundary, it can be proved (see \cite[Sect.~4.3]{grigoryan}) that
\begin{equation}
\label{eq:capint}
\C^h(K,\Omega)=\int_{\Omega} |\nabla u|^{2}\,dv_h=\int_{\partial K}|\nabla u|\,da_h=\int_{\ptl K}\frac{\ptl u}{\ptl\nu}\,da_h,
\end{equation}
where $\nu$ is the outer unit normal along $\ptl(\Om-K)$, i.e., the unit normal along $\ptl K$ pointing into $K$, and $u$ is the solution of the following Dirichlet problem for the weighted Laplace equation
\begin{equation}
\label{laplace}
\begin{cases}
\Delta^h u=0 \quad \text{in } \Om-K, \\
u\!\mid_{\partial K}=1,\\
u\mid_{\partial\Omega}=0.
\end{cases}
\end{equation}
Hence, the infimum in the definition of $\C^h(K,\Omega)$ is attained
by the solution to \eqref{laplace}. This function $u$ is called the
 \emph{$h$-capacity potential} of the capacitor $(K,\Omega)$. 

The relation between weighted capacities and the $h$-parabolicity (resp. $h$-hyperboli\-city) of a weighted manifold is shown in the next result, see \cite{Gri3}.

\begin{theorem}
\label{theorGrig}
Let $(M^m,g,e^h)$ be a weighted manifold. Then, $M$ is
$h$-parabolic $($resp. $h$-hyperbolic$)$ if and only if $M$ has null $($resp. positive$)$ $h$-capacity, i.e., there exists a precompact open set $D\subseteq M$ such that $\C^h(D)=0$ $($resp. $\C^h(D) >0$$)$.
\end{theorem}

In view of Theorem~\ref{theorGrig} and equality \eqref{eq:sensible}, in order to determine the $h$-parabolicity or $h$-hyperbolicity of a weighted manifold it suffices to find bounds on $\C^h(D,\Om_i)$ for some set $D\subset M$ and some exhaustion $\{\Om_i\}_i^\infty$. This will be done by assuming suitable bounds on the weighted or unweighted curvatures of the manifold.

\subsection{Weighted curvatures}\
\label{subsec:wc}

Let us consider a complete weighted manifold $(M^m,g,e^h)$. In this subsection we recall different notions of curvature in this context. 

The most extended generalizations of the Ricci curvature tensor are the Bakry-\'Emery Ricci tensors. These were introduced by Lichnerowicz \cite{lich1,lich2} and later employed by Bakry and \'Emery \cite{be} in the framework of diffusion generators.

\begin{definition}
The \emph{$\infty$-Bakry-\'Emery Ricci tensor} in $(M,g,e^h)$ is the $2$-tensor
$${\rm Ric}^h_\infty:={\rm Ric} -{\rm Hess}\, h,$$
where $\text{Ric}$ and $\text{Hess}$ denote the Ricci tensor and the Hessian in $(M,g)$. For any $q>0$, the \emph{$q$-Bakry-\'Emery Ricci tensor} in $(M,g,e^h)$ is defined as
$${\rm Ric}^h_q:={\rm Ric} -{\rm Hess}\, h -\frac{1}{q}\,\nabla h\otimes \nabla h.$$
\end{definition}
Observe that
$${\rm Ric}^h_\infty={\rm Ric}^h_q +\frac{1}{q}\,\nabla h\otimes \nabla h,$$ so that a lower bound on ${\rm Ric}^h_q$ implies the same lower bound on ${\rm Ric}^h_\infty$.

In this work the Bakry-\'Emery Ricci tensors will be used to deduce  inequalities for the weighted Laplacian of the distance function $r$ from a fixed point $o\in M$. On the other hand, comparison results for the Hessian of $r$ will be obtained by assuming bounds on the weighted sectional curvatures, that we now introduce.

\begin{definition}
\label{Def-Sec-h}
Fix a point $o \in M$. We denote by $r:M\to [0,\infty[$ the distance function from $o$ in $(M,g)$, and by $\text{cut}(o)$ the cut locus of $o$ in $(M,g)$. It is well known that $r$ is smooth on $M-(\text{cut}(o)\cup\{o\})$. For any point $p\in M-(\text{cut}(o)\cup\{o\})$, and any plane $\sg_p\subeq T_pM$ containing the radial direction $(\nabla r)_p$, we define the \emph{$\infty$-weighted sectional curvature} of $\sg_p$ as
\begin{displaymath}
{\rm Sec}_{\infty}^h(\sigma_p):={\rm
Sec}(\sigma_p)-\frac{1}{m-1}\,(\Hess h)_p\big((\nabla r)_p,(\nabla r)_p\big),
\end{displaymath}
where \text{Sec} stands for the sectional curvature in $(M^m,g)$. On the other hand, for any $q>0$, we define the {\it $q$-weighted sectional curvature} of $\sigma_p$ by
\begin{displaymath}
{\rm Sec}_{q}^h(\sigma_p):=\text{Sec}^h_\infty(\sg_p)-\frac{1}{(m-1)\,q}\,(\nabla h\otimes\nabla h)\big((\nabla r)_p,(\nabla r)_p\big).
\end{displaymath}
\end{definition}

We remark that, up to some constants, the previous definitions coincide with the ones introduced by Wylie~\cite{W}. Note also that, if $\{x_1,\ldots,x_{m-1}\}\subset T_pM$ is any orthonormal basis orthogonal to $(\nabla r)_p$, then for $q>0$ or $q=\infty$, we have
\[
\sum_{i=1}^{m-1}\text{Sec}^h_q(\sg_i)=(\text{Ric}^h_q)_p((\nabla r)_p,(\nabla r)_p),
\]
where $\sg_i$ denotes the plane spanned by $\{x_i,(\nabla r)_p\}$. Hence, a lower bound on $\text{Sec}^h_q$ implies  the same lower bound multiplied by $m-1$ on $\text{Ric}^h_q(\nabla r,\nabla r)$.

\subsection{Weighted model spaces}\ 
\label{secModel}

Here we introduce the model spaces that we will use to establish our comparison theorems.
\begin{definition}(see {\cite[Ch.~2]{GreW}, \cite[Sect.~3]{grigoryan}, \cite[Ch.~3]{Pe}}).
\label{model}
A $w$-model space is a smooth warped product $(M^m_w,g_w):=B^1\times_w F^{m-1}$ with base $B^{1}:=[0,\Lambda[\,\subset\mathbb{R}$ (where $0 < \Lambda \leq \infty$), fiber $F^{m-1}:=\mathbb{S}^{m-1}_{1}$ (the unit $(m-1)$-sphere with standard metric), and warping function $w:[0,\Lambda[\,\,\to [0,\infty[$ such that $w(r)>0$ for all $r>0$, whereas $w(0) = 0$, $w'(0) = 1$, and $w^{k)}(0) = 0$ for all even derivation orders. The point $o_{w}:=\pi^{-1}(0)$, where $\pi$ denotes the projection onto $B^1$, is called the {\em{center point}} of the
model space. If $\Lambda = \infty$, then $o_{w}$ is a pole of the manifold (recall that a \emph{pole} of a complete Riemannian manifold $M$ is a point $o\in M$ such that the exponential map $\exp_o:T_oM\to M$ is a diffeomorphism).
\end{definition}

\begin{remark}
The analytical conditions for $w$ at $0$ ensure that the warped metric is $C^\infty$ at the pole. However, for most of our comparison results it suffices to require $w(0)=0$ and $w'(0)=1$. In this case we get a model space with less regularity at the pole that we still denote $(M^m_w,g_w)$.
\end{remark}

\begin{example}
\label{propSpaceForm}
The simply connected space forms $\kam$ of constant sectional curvature $b$
can be constructed as $w$-models with any given point as center
point using the warping functions
\begin{equation*}
w_{b}(r):=
\begin{cases} 
\frac{1}{\sqrt{b}}\sin(\sqrt{b}\, r) &\text{if $b>0$},
\\
\phantom{\frac{1}{\sqrt{b}}} r &\text{if $b=0$},
\\
\frac{1}{\sqrt{-b}}\sinh(\sqrt{-b}\,r) &\text{if $b<0$}.
\end{cases}
\end{equation*}
Note that, for $b > 0$, the function $w_{b}(r)$ admits a smooth
extension to  $r = \pi/\sqrt{b}$. For $\, b \leq 0\,$ any center
point is a pole.
\end{example}

In \cite{GreW,grigoryan,MP3,MP4,oneill} we have a complete
description of the $w$-model spaces. In particular, the
sectional curvatures for planes containing the radial direction from the
center point are determined by the radial function
\[
-\frac{w''(r)}{w(r)}.
\] 
Moreover, the mean curvature of the metric sphere of radius $r$ from the center $o_w$ is
\begin{equation*}
\eta_{w}(r):=\frac{w'(r)}{w(r)}=\frac{d}{dr}\ln(w(r)) .
\end{equation*}

A \emph{weighted $(w,f)$-model space} is a triple $(M^m_w, g_w, e^{f(r)})$ where $e^{f(r)}$ is a radial weight in the $w$-model ($M^m_w,g_w)$. In this situation, the weighted volumes of the open metric ball $B^w_R$ of radius $R>0$ centered at $o_w$, and of the sphere $\ptl B^w_R$ are computed as follows
\begin{equation*}
\begin{split}
\Vol_f(B^w_R)&=V_0\,\int_0^R w^{m-1}(t)\,e^{f(t)}\,dt,
\\
\Vol_f(\partial B^w_R)&=V_0\,w^{m-1}(R)\,e^{f(R)},
\end{split}
\end{equation*}
where $V_0$ is the Riemannian volume of the unit sphere $\mathbb{S}^{m-1}_1$. We will denote by $q_{w,f}$ the \emph{weighted
isoperimetric quotient} for balls around the center, defined by
\[
q_{w,f}(R):=\frac{\Vol_f(B^w_R)}{\Vol_f(\ptl B^w_R)}=\frac{\int_0^R w^{m-1}(t)\,e^{f(t)}\,dt}{w^{m-1}(R)\,e^{f(R)}}.
\]

In \cite{HPR1} we computed the weighted capacity $\C^f(B^w_\rho,B^w_R)$ for any two radii $\rho,R>0$ with $R>\rho$. By equations \eqref{eq:capint} and \eqref{laplace} this is determined by the associated $f$-capacity potential, i.e., the solution to the following weighted Dirichlet problem
\begin{equation}
\label{eqfDirModel2}
\begin{cases}
\Delta_{M_w}^f u = 0\,\,\,&\text{in\, $A^w_{\rho,R}$},\\
\phantom{\Delta }u = 1\,\,\,&\text{in\, $\partial B^w_\rho$}, \\
\phantom{\Delta }u = 0\,\,\,&\text{in\, $\partial B^w_R$},
\end{cases}
\end{equation}
where $A^w_{\rho,R}$ is the annulus $B^w_R-\overline{B}^w_\rho$ in $M^m_w$. For later use we must mention that a radial function $\phi(r)$ defined on $A^w_{\rho,R}$ satisfies the first equation in \eqref{eqfDirModel2} if and only if
\begin{equation}
\label{eq:sabika}
\phi''(r)+\phi'(r)\,\bigg((m-1)\,\frac{w'(r)}{w(r)}+f'(r)\bigg)=0.
\end{equation}

\begin{theorem}[\cite{HPR1}]
\label{capacityweighted} 
In a weighted $(w,f)$-model space $(M^m_w,g_w,e^{f(r)})$, the solution to the Dirichlet problem \eqref{eqfDirModel2} in the annulus $A^w_{\rho,R}$ is given by the radial function
\begin{equation}
\label{fsolmodel}
\phi_{\rho,R,f}(r):=\left(\int_r^R w^{1-m}
(s)\,e^{-f(s)}\,ds\right)\,\left(\int_\rho^R w^{1-m}
(s)\,e^{-f(s)}\,ds\right)^{-1}.
\end{equation}
Therefore, we have
\begin{equation}
\begin{split}
\label{fcapacitymodel}
\C^f(B^w_\rho,B^w_R)&=|\phi_{\rho,R,f}'(\rho)|\,\Vol_f
(\partial B^w_\rho
)\\
&=V_0\,\left(\int_\rho^R w^{1-m}(s)\,e^{-f(s)}\,ds\right)^{-1}.
\end{split}
\end{equation}
\end{theorem}

\begin{remark} 
\label{Alhfors} 
The last equality in equation ~\eqref{fcapacitymodel} can be written in terms of the weighted area of the geodesic spheres, so that we get
\begin{displaymath}
\C^f(B^w_\rho,B^w_R)=\left(\int_\rho^R\frac{dt}{\Vol_f(\partial B^w_t)}\right)^{-1}.
\end{displaymath}
As a direct consequence, it can be obtained a weighted version of the Ahlfors criterion: a weighted $(w,f)$-model space is $f$-parabolic if and only if $\int_\rho^\infty \Vol_f(\ptl B^w_t)^{-1}\,dt=\infty$ for some $\rho>0$, see \cite{Gri2,HPR1}. 
\end{remark}

\section{Analysis of the distance function}
\label{distanceanalysis}

Let $(M^m,g,e^h)$ be a weighted manifold such that $M$ is complete and noncompact. In this section, depending on lower or upper bounds for some of the weighted curvatures, we provide Laplacian and Hessian comparisons for the distance function $r:M\to[0,\infty[$ from a fixed point $o\in M$. Then, we will deduce estimates outside the cut locus $\text{cut}(o)$ for the Hessian and Laplacian of a modification $F\circ r$ associated to a smooth function $F:(0,\infty)\to\rr$.

While the Laplacian comparisons in Subsection \ref{subLap} will lead us to intrinsic comparison results for the volume and the capacity of balls centered at $o$, which eventually provide an intrinsic description of parabolicity, the Hessian comparisons in Subsection \ref{subHess} will allow us to deduce analogous consequences in the extrinsic setting, i.e., when we consider a submanifold $P$ of the ambient weighted manifold $M$.

The starting point for our estimates is an inequality which comes from the Index Lemma and the relation between the Hessian of the distance function $r$ and the index form over Jacobi vector fields.

Fix a point $p\in M-({\rm cut}(o)\cup \{o\})$ and denote $r_p:=r(p)$. Let
$\gamma:[0,r_p]\longrightarrow M$ be the minimizing geodesic parameterized by arc-length joining $o$ with $p$. Take a unit vector $x\in T_p M$ with $x\perp (\nabla r)_p$. It is well known, see \cite[Ch.~2]{GreW}, that
\[
(\Hess r)_p(x,x)=I_\gamma(J,J):=\int_0^{r_p} \left(|J'|^2-\langle {\rm R}(J,\gamma')J,\gamma'\rangle\right)(t)\,dt,
\]
where $J$ is the Jacobi vector field along $\gamma$ with
$J(0)=0$ and $J(r_p)=x$. Here $J'$ stands for the covariant derivative of $J$, and $\text{R}$ is the Riemann curvature tensor in $(M,g)$. On the other hand, the Index Lemma \cite[Sect.~10.2]{dcriem} implies that
\[
I_\gamma(J,J)\leq I_\gamma (X,X),
\]
for any vector field $X$ along $\gamma$ such that $X(0)=J(0)$ and $X(r_p)=J(r_p)$. As a consequence
\begin{displaymath}
(\Hess r)_p(x,x)\leq \int_0^{r_p} \left(|X'|^2-\langle{\rm R}(X,\gamma')X,\gamma'\rangle\right)(t)\,dt,
\end{displaymath}
for any vector field $X$ along $\gamma$ with $X(0)=0$ and $X(r_p)=x$.

Let $w(s)$ be a function such that $w(0)=0$ and $w(s)>0$ for all $s>0$. We define $X(t):=\frac{w(t)}{w(r_p)}\,Y(t)$, where $Y$ is the unique parallel vector
field along $\gamma$ with $Y(r_p)=x$. Since $|X(t)|^2=\frac{w^2(t)}{w^2(r_p)}$, $\langle X(t),\gamma'(t)\rangle=0$ and $X'(t)=\frac{w'(t)}{w(r_p)}\,Y(t)$, we deduce the inequality
\begin{equation}
\label{eq:seque}
\begin{aligned}
(\Hess r)_p(x,x) &\leq \frac{1}{w^2(r_p)}\, \int_0^{r_p}
\left((w')^2(t)-{\rm
Sec}(\sigma_{\gamma(t)})\,w^2(t)\right)dt,
\end{aligned}
\end{equation}
where $\sigma_{\gamma(t)}\subseteq T_{\gamma(t)}M$ is the plane spanned by $\{\gamma'(t),X(t)\}$.

\subsection{Laplacian comparisons under lower bounds for the Ricci curvatures}\
\label{subLap}

In the next result we obtain inequalities for the weighted Laplacian of $r$ that generalize previous estimates for the unweighted case $h=0$ given in \cite[Ch.~2]{GreW}, see also \cite{Pa2,Zhu}.

\begin{theorem}[\cite{qian,WW,MRS,PRRS}]
\label{th:laplaceineqr} 
Let $(M^m,g, e^h)$ be a weighted manifold, $r:M\rightarrow[0,\infty[$ the distance function from a fixed point $o\in M$, and $w(s)$ a smooth function such that $w(0)=0$ and $w(s)>0$ for all $s>0$.
\begin{itemize}
\item[a) ] If there is $q>0$ such that the $q$-Bakry-\'Emery Ricci curvature in the radial direction is bounded from below in $M-(\rm{cut}(o)\cup\{o\})$ as 
 $${\rm Ric}^h_q (\nabla r,\nabla r) \geq
-(m+q-1)\,\frac{w''(r)}{w(r)},$$
then
\begin{equation}
\label{lap1}
\Delta^hr\leq (m+q-1)\,\frac{w'(r)}{w(r)}
\end{equation}
on $M-(\rm{cut}(o)\cup\{o\})$. 

As a consequence, for every smooth function $F:(0,\infty)\to\rr$ with $F'\geq 0$ $($respectively $F'\leq 0$$)$, we have
\begin{equation}
\label{lapF1} 
\Delta^h (F\circ r)\leq (\geq)\,F''(r)+F'(r)\,(m+q-1)\,\frac{w'(r)}{w(r)}
\end{equation}
on $M-(\rm{cut}(o)\cup\{o\})$. 

\medskip

\item [b) ] If the $\infty$-Bakry-\'Emery Ricci curvature in the radial direction is bounded from below in $M-(\rm{cut}(o)\cup\{o\})$ as 
$${\rm Ric}^h_\infty(\nabla r,\nabla r) \geq
-(m-1)\,\frac{w''(r)}{w(r)},$$ 
and there exists a non-decreasing $C^1$ function $\theta:[0,\infty[\to\rr$ such that
$$\langle\nabla h,\nabla r\rangle\,w'(r)\leq
\theta(r)\,w'(r)$$
on $M-(\rm{cut}(o)\cup\{o\})$, then
\begin{equation}
\label{lap2}
\Delta^h r\leq (m-1)\,\frac{w'(r)}{w(r)}+\theta(r)
\end{equation}
on $M-(\rm{cut}(o)\cup\{o\})$.

As a consequence, for every smooth function $F:(0,\infty)\to\rr$ with $F'\geq 0$ $($respectively $F'\leq 0$$)$, we have
\begin{equation}
\label{lapF2}
\Delta^h (F\circ r)\leq (\geq)\,F''(r)+\,F'(r)\left((m-1)\,\frac{w'(r)}{w(r)}+\theta(r)\right)
\end{equation}
on $M-(\rm{cut}(o)\cup\{o\})$. 
\end{itemize}
\end{theorem}

\begin{proof}
The proof of \eqref{lap1} can be found in \cite{qian} when the lower bound on $\text{Ric}^h_q$ is constant, and in \cite{MRS} for the general
case. The proof of \eqref{lap2} appears in \cite{WW} for constant bounds, and in \cite{PRRS} when $\frac{w''(r)}{w(r)}$ is a positive function and $\theta:[0,\infty[\to\rr$ is continuous and nondecreasing. The starting point for both proofs is the Bochner-Weitzenb\"ock formula in weighted manifolds. We provide below a complete proof of \eqref{lap2} based on the inequality \eqref{eq:seque}.

Take a point $p \in M - ({\rm cut}(o) \cup \{o\})$ and denote $r_p:=r(p)$. Let $\gamma:[0,r_p]\to M$ be the minimizing geodesic parameterized by arc-length joining $o$ with $p$. It is well known that $\gamma'(t)=(\nabla r)_{\gamma(t)}$ for any $t\neq 0$. We apply \eqref{eq:seque} to an orthonormal family $\{x_1,\ldots,x_{m-1}\}$ in $T_pM$ orthogonal to $(\nabla r)_p$. By summing up, and taking into account that $(\text{Hess}\,r)_p((\nabla r)_p,(\nabla r)_p)=0$, it follows that
\[
(\Delta r)(p)\leq\frac{1}{w^2(r_p)}\int_0^{r_p}\big((m-1)\,(w')^2(t)-\text{Ric}_{\gamma(t)}((\nabla r)_{\gamma(t)},(\nabla r)_{\gamma(t)})\,w^2(t)\big)\,dt.
\]

Define the function $f(t):=(h\circ\gamma)(t)$. It is clear that $f'(t)=\escpr{\nabla h,\nabla r}(\gamma(t))$ and $f''(t)=(\text{Hess}\,h)_{\gamma(t)}((\nabla r)_{\gamma(t)},(\nabla r)_{\gamma(t)})$. Since $\text{Ric}=\text{Ric}^h_\infty+\text{Hess}\,h$, by using the lower bound on $\text{Ric}^h_\infty$ and integration by parts, we get
\begin{align*}
(\Delta r)(p)&\leq\frac{m-1}{w^2(r_p)}\int_0^{r_p}\big((w')^2(t)+w(t)\,w''(t)\big)\,dt-\frac{1}{w^2(r_p)}\int_0^{r_p}w^2(t)\,f''(t)\,dt
\\
&=(m-1)\,\frac{w'(r_p)}{w(r_p)}-\frac{1}{w^2(r_p)}\int_0^{r_p}w^2(t)\,f''(t)\,dt.
\end{align*}
On the other hand, the definition of weighted Laplacian in \eqref{eq:flaplacian} gives us
\[
(\Delta^h\,r)(p)=(\Delta r)(p)+\escpr{\nabla h,\nabla r}(p)=(\Delta r)(p)+f'(r_p).
\]
By substituting this information above and integrating by parts, we obtain
\begin{align*}
(\Delta^h r)(p)&\leq (m-1)\,\frac{w'(r_p)}{w(r_p)}+\frac{1}{w^2(r_p)}\int_0^{r_p}
(w^2)'(t)\,f'(t)\,dt
\\
&\leq (m-1)\,\frac{w'(r_p)}{w(r_p)}+\frac{1}{w^2(r_p)}\int_0^{r_p} (w^2)'(t)\,\theta(t)\,dt
\\
&=(m-1)\,\frac{w'(r_p)}{w(r_p)}+\theta(r_p)-\frac{1}{w^2(r_p)}\int_0^{r_p} w^2(t)\,\theta'(t)\,dt
\\
&\leq (m-1)\,\frac{w'(r_p)}{w(r_p)}+\theta(r_p),
\end{align*}
where we have used the hypothesis $f'(t)\,w'(t)\leq \theta(t)\,w'(t)$ and that $\theta(t)$ is a non-decreasing function. This shows \eqref{lap2}.

Finally, to deduce \eqref{lapF1} and \eqref{lapF2} it suffices to have in mind the estimates in \eqref{lap1} and \eqref{lap2} together with equality
\begin{equation}
\label{eq:bullin}
\Delta\,(F\circ r)=F''(r)+F'(r)\,\Delta r,
\end{equation}
which holds on $M-(\text{cut}(o)\cup\{o\})$.
\end{proof}

\begin{remark} 
When we assume additional conditions on the derivatives $w^{k)}(0)$, we can identify the lower bounds in Theorem \ref{th:laplaceineqr} for the radial Bakry-\'Emery Ricci curvatures and the upper bounds deduced for $\Delta^h r$ as the radial Ricci curvatures and the Laplacian of the distance function from $o_w$ in some $w$-model space. 

More precisely, if $q \in \mathbb{N}$ in statement a), then $-(m+q-1)\,\frac{w''(r)}{w(r)}$ is the Ricci curvature in the radial direction of the model $(M^{m+q}_w,g_w)$, whereas $(m+q-1)\,\frac{w'(r)}{w(r)}$ is the Laplacian in $(M^{m+q}_w,g_w)$ of the distance function from the center point $o_w$. 

In a similar way, the bound in statement b) for $\text{Ric}_\infty^h$ coincides with the radial Ricci curvature of the model space $(M^{m}_w,g_w)$, whereas the quantity $(m-1)\,\frac{w'(r)}{w(r)} + \theta(r)$ coincides with the weighted
Laplacian of the distance function in the weighted $(w,f)$-model space
$(M^m_w,g_w,e^{f(r)})$ with $f(r):=\int_0^r\theta(s)\,ds$. 
\end{remark}

\subsection{Hessian comparisons under lower bounds for the sectional curvatures}
\label{subHess}\

Here we follow the ideas employed in the unweighted case \cite[Ch.~2]{GreW} to establish comparisons for the Hessian of the distance function involving the radial derivatives of the weight and the weighted sectional curvatures introduced in Definition~\ref{Def-Sec-h}.

\begin{theorem}
\label{th:hessianineqr} 
Let $(M^m,g, e^h)$ be a weighted manifold, $r:M\rightarrow
[0,\infty[$ the distance function from a point $o\in M$, and $w(s)$ a smooth function such that $w(0)=0$ and $w(s)>0$ for all $s>0$.
\begin{itemize}
\item[a) ] Suppose that there is $q>0$ such that, for any $p\in M-({\rm cut}(o)\cup \{o\})$ and any plane $\sigma_p\subseteq T_pM$ containing the radial direction $(\nabla r)_p$, the $q$-weighted sectional curvature is bounded from below as
\begin{displaymath}
{\rm Sec}_{q}^h(\sigma_p)\geq -\,\frac{m+q-1}{m-1}\,\frac{w''(r)}{w(r)}.
\end{displaymath}
Then, the inequality
\begin{equation}
\label{hess1}
(\Hess r)(x,x)+\frac{1}{m-1}\,\langle \nabla h,\nabla
r\rangle \leq \frac{m+q-1}{m-1}\,\frac{w'(r)}{w(r)}
\end{equation}
holds on $M-({\rm cut}(o)\cup \{o\})$ for any unit tangent vector $x$ orthogonal to $\nabla r$.

As a consequence, for every smooth function $F:(0,\infty)\to\rr$ with $F'\geq 0$ $($respectively $F'\leq 0$$)$, we obtain
\begin{equation}
\label{hessF1}
\Delta^h (F\circ r)\leq (\geq)\,F''(r)+F'(r)\,(m+q-1)\,\frac{w'(r)}{w(r)}
\end{equation}
on $M-({\rm cut}(o)\cup \{o\})$.

\medskip

\item[b) ] If, for any $p\in M-({\rm cut}(o)\cup \{o\})$ and any plane $\sigma_p\subseteq T_pM$ containing $(\nabla r)_p$, the $\infty$-weighted sectional curvature is bounded from below as
\begin{equation*}
{\rm Sec}_{\infty}^h(\sigma_p) \geq -\,\frac{w''(r)}{w(r)},
\end{equation*}
and there exists a non-decreasing $C^1$ function $\theta:[0,\infty[\to\rr$ such that
$$\langle\nabla h,\nabla r\rangle\,w'(r)\leq
\theta(r)\,w'(r)$$
on $M-({\rm cut}(o)\cup \{o\})$, then the inequality
\begin{equation}
\label{hessianineqinfty}
(\Hess r)(x,x) +\frac{1}{m-1}\,\langle \nabla h,\nabla r\rangle\leq\frac{w'(r)}{w(r)}+\frac{1}{m-1}\,\theta(r)
\end{equation}
holds on $M-({\rm cut}(o)\cup \{o\})$ for any unit tangent vector $x$ orthogonal to $\nabla r$.

As a consequence, for every smooth function $F:(0,\infty)\to\rr$ with $F'\geq 0$ $($respectively $F'\leq 0$$)$, we obtain
\begin{equation}
\label{hessF2}
\Delta^h (F\circ r)\leq (\geq)\,F''(r)+F'(r)\left((m-1)\,\frac{w'(r)}{w(r)}+\theta(r)\right)
\end{equation}
on $ M-({\rm cut}(o)\cup \{o\})$.
\end{itemize}
\end{theorem}

\begin{proof} 
Fix a point $p \in M - ({\rm cut}(o) \cup \{o\})$ and denote $r_p:=r(p)$. Let $\gamma:[0,r_p]\to M$ be the minimizing geodesic parameterized by arc-length joining $o$ with $p$. It is well known that $\gamma'(t)=(\nabla r)_{\gamma(t)}$ for any $t\neq 0$. Define the function $f(t):=(h\circ\gamma)(t)$. Note that $f'(t)=\escpr{\nabla h,\nabla r}(\gamma(t))$ and $f''(t)=(\text{Hess}\,h)_{\gamma(t)}((\nabla r)_{\gamma(t)},(\nabla r)_{\gamma(t)})$.

Let us prove \eqref{hess1}. Starting from \eqref{eq:seque}, and having in mind the lower bound for $\text{Sec}^h_q(\sigma_{\gamma(t)})$, we arrive at
\begin{equation}
\label{ineq:indexform}
\begin{split}
(\Hess r)_p(x,x)&=\frac{1}{w^2(r_p)} \int_0^{r_p} (w')^2(t)\,dt
\\
&-\frac{1}{w^2(r_p)} \int_0^{r_p}\left({\rm Sec}_{q}^h(\sigma_{\gamma(t)})+\frac{f''(t)}{m-1}+\frac{(f')^2(t)}{(m-1)\,q}\right)\,w^2(t)\,dt 
\\ 
&\leq\frac{1}{w^2(r_p)}\,\int_0^{r_p}\left((w')^2(t)
+\frac{m+q-1}{m-1}\,w''(t)\,w(t)\right)dt
\\
&-\frac{1}{w^2(r_p)}\,\int_0^{r_p}\left(\frac{f''(t)\,w^2(t)}{m-1}+\frac{(f')^2(t)\,w^2(t)}{(m-1)\,q}\right)dt. 
\end{split}
\end{equation}
On the other hand, applying integration by parts, we get
\begin{equation}
\label{int_parts}
\int_0^{r_p} f''(t)\,w^2(t)\,dt=f'(r_p)\,w^2(r_p)-\int_0^{r_p} 2\,f'(t)\,w(t)\,w'(t)\,dt.
\end{equation}
Moreover
\begin{equation}
\label{newton}
2\,f'(t)\,w(t)\,w'(t)\leq\frac{1}{q}\,(f'\,w)^2(t)+q\,(w')^2(t).
\end{equation}
Replacing the information of \eqref{int_parts} and \eqref{newton} into \eqref{ineq:indexform}, we obtain 
\[
(\Hess r)_p(x,x)\leq\frac{m+q-1}{(m-1)\,w^2(r_p)}\,\int_0^{r_p}\left((w')^2(t)+w''(t)\,w(t)\right)dt-\frac{1}{m-1}\,f'(r_p).
\]
Finally, integrating by parts again, it follows that
\begin{equation*}
 (\Hess r)_p(x,x) + \frac{1}{m-1}\,f'(r_p)\leq\frac{m+q-1}{m-1}\,\frac{w'(r_p)}{w(r_p)},
\end{equation*}
which is the desired inequality at the point $p$.

Let us prove \eqref{hessianineqinfty}. Starting from \eqref{eq:seque}, and taking into account the lower bound for $\text{Sec}^h_\infty(\sigma_{\gamma(t)})$, we get
\begin{align*}
(\text{Hess}\,r)_p(x,x)&\leq\frac{1}{w^2(r_p)}\int_0^{r_p}\left((w')^2(t)+w''(t)\,w(t)\right)dt
\\
&-\frac{1}{(m-1)\,w^2(r_p)}\int_0^{r_p}f''(t)\,w^2(t)\,dt.
\end{align*}
Using integration by parts, the hypothesis $f'(t)\,w'(t)\leq
\theta(t)\,w'(t)$, and the fact that $\theta'\geq 0$, we obtain
\begin{equation*}
\label{hessian-ineq}
\begin{split}
(\Hess r)_p(x,x)&\leq \frac{w'(r_p)}{w(r_p)}-\frac{f'(r_p)}{(m-1)}+\frac{1}{(m-1)\,w^2(r_p)}\int_0^{r_p}(w^2)'(t)\,f'(t)\,dt
\\
&\leq\frac{w'(r_p)}{w(r_p)}-\frac{f'(r_p)}{m-1}+\frac{1}{(m-1)\,w^2(r_p)}\int_0^r (w^2)'(t)\,\theta(t)\,dt
\\
&=\frac{w'(r_p)}{w(r_p)}+\frac{\theta(r_p)-f'(r_p)}{m-1}-\frac{1}{(m-1)\,w^2(r_p)}\int_0^r w^2(t)\,\theta'(t)\,dt
\\
&\leq \frac{w'(r_p)}{w(r_p)}-\frac{1}{m-1}\,f'(r_p)+\frac{1}{m-1}\,\theta(r_p).
\end{split}
\end{equation*}
This is equivalent to inequality \eqref{hessianineqinfty} at the point $p$.

Finally, inequalities \eqref{hessF1} and \eqref{hessF2} follow from \eqref{hess1}, \eqref{hessianineqinfty} and the definition of weighted Laplacian in \eqref{eq:flaplacian}, with the help of the identities
\begin{align*}
\text{Hess}\,(F\circ r)(X,X)&=F''(r)\,\escpr{\nabla r,X}^2+F'(r)\,(\text{Hess}\,r)(X,X),
\\
(\text{Hess}\,r)(\nabla r,\nabla r)&=0,
\end{align*}
where $X$ is any vector field on $M-(\text{cut}(o)\cup\{o\})$.
\end{proof}

\begin{remark}
\label{invhyp1}
By assuming additional conditions on the derivatives $w^{k)}(0)$ the lower bounds in Theorem~\ref{th:hessianineqr} for the weighted sectional curvatures of radial planes and the upper bounds for $\text{Hess}\,r$ are related to the sectional curvatures of the radial planes and the Hessian of the distance function from $o_w$ in some $w$-model space.  Indeed, in  $(M^m_w,g_w)$ the radial sectional curvatures equal $-\frac{w''(r)}{w(r)}$, whereas the value of $\text{Hess}\,r$ along any unit direction orthogonal to the radial one is $\frac{w'(r)}{w(r)}$.
\end{remark}

\begin{remark}
The comparisons for $\text{Hess}\,r$ in Theorem~\ref{th:hessianineqr} extend in $M-(\text{cut}(o)\cup\{o\})$ for any tangent vector $y$. Write $y=x+\la\,(\nabla r)_p$, where $x\in T_pM$ is orthogonal to $(\nabla r)_p$ and $\la:=\escpr{y,(\nabla r)_p}$. Since $(\text{Hess}\,r)_p(u,(\nabla r)_p)=0$ for any vector $u\in T_pM$, it follows that
\[
(\text{Hess}\,r)_p(y,y)=(\text{Hess}\,r)_p(x,x),
\]  
so that an estimate for $(\text{Hess}\,r)_p(x,x)$ leads to an estimate for $(\text{Hess}\,r)_p(y,y)$. In the particular case of \eqref{hess1}, we deduce that the inequality
\begin{equation}
\label{eq:setenil}
(\text{Hess}\,r)(y,y)\leq\big(|y|^2-\escpr{y,\nabla r}^2\big)\bigg(\frac{m+q-1}{m-1}\,\frac{w'(r)}{w(r)}-\frac{1}{m-1}\,\escpr{\nabla h,\nabla r}\bigg)
\end{equation}
holds on $M-(\text{cut}(o)\cup\{o\})$ for any tangent vector $y$.
\end{remark}

\section{Intrinsic comparison results}
\label{intcompan}

In this section we consider a complete non-compact weighted manifold and present three series of results where, assuming lower or upper bounds for some of the weighted or unweighted curvatures of the manifold, we provide estimates for the weighted volumes and capacities of metric balls about a given point. From the capacity estimates we will deduce conclusions about the parabolicity or hyperbolicity of the manifold.

Along this section we will denote by $B_R$ (resp. $B^w_R$) the open metric ball of radius $R>0$ centered at a fixed point $o\in M^m$ (resp. at the pole $o_w\in M^m_w$).

\subsection{Comparisons under a lower bound on the $\infty$-Bakry-\'Emery Ricci curvatures}
\label{infiniteBakry}\

The first result of this section shows that, as happens in the Riemannian setting, estimates for the weighted Laplacian of the distance function allow to establish bounds for weighted isoperimetric quotients and volumes of metric balls. Our proof goes in the line of  \cite{Pa,MaP1,MP5}, where comparisons for the unweighted isoperimetric quotient of extrinsic balls of submanifolds were obtained. Previous comparisons involving weighted volumes and quotients of weighted volumes (but not weighted isoperimetric quotients) when $\text{Ric}^h_\infty$ is bounded from below can be found in \cite{qian2,lott,morganmyers,WW,PRRS,MR3197658}. 

\begin{theorem} 
\label{th:isoper_infty}
Let $(M^m,g,e^h)$ be a complete and non-compact weighted manifold, $r:M\rightarrow [0,\infty[$ the distance function from a point $o\in M$, and $w(s)$ a smooth function with $w(0)=0$, $w'(0)=1$ and $w(s)>0$ for all $s>0$.  Suppose that the following conditions are fulfilled:
\begin{itemize}
\item [a) ] Every radial $\infty$-Bakry-\'Emery Ricci
curvature is bounded as
\[
{\rm Ric}^h_\infty(\nabla r,\nabla r) \geq -(m-1)\,\frac{w''(r)}{w(r)} \ \text{ on }M-(\rm{cut}(o)\cup\{o\}).
\]
\item [b) ] There exists a non-decreasing $C^1$ function $\theta:[0,\infty[\to\rr$ such that  
$$\langle\nabla h,\nabla r\rangle\,w'(r)\leq \theta(r)\,w'(r) \ \text{ on }M-(\rm{cut}(o)\cup\{o\}).$$ 
\end{itemize}
Then, we have
\begin{equation} 
\label{ineq:quotient}
\frac{\Vol_h(B_R)}{\Vol_h(\partial B_R)} \geq
\frac{\Vol_{f}(B^w_R)}{\Vol_{f}(\partial B^w_R)}, \ \text{ for almost any }R>0,
\end{equation}
where $\Vol_{f}$ stands for the weighted volume in the $(w,f)$-model space $(M^m_w,g_w,e^{f(r)})$ with $f(r):=\int_0^r\theta(s)\,ds$. As a consequence
\begin{align}
\label{ineq:volume}
\Vol_h(B_R)&\leq e^{h(o)}\,\Vol_{f} (B^w_R), \ \text{ for any }R>0,
\\
\label{ineq:volume2}
\Vol_h(\partial B_R)&\leq e^{h(o)}\,\Vol_{f} (\partial B^w_R), \ \text{ for almost any }R>0.
\end{align}
In particular, if $\Vol_f(M_w)<\infty$, then $\Vol_h(M)<\infty$.
\end{theorem}

\begin{proof}
We will prove the comparison for the weighted isoperimetric quotient by using the function solution of a Dirichlet-Poisson problem. 

Let $q_{w,f}:(0,\infty)\to\rr$ be the weighted isoperimetric quotient in $(M_w^m,g_w,e^{f(r)})$, which is given by
\[
q_{w,f}(t):=\frac{\Vol_f(B^w_t)}{\Vol_f(\ptl B^w_t)}=
\frac{\int_0^t w^{m-1}(s)\,e^{f(s)}\,ds}{w^{m-1}(t)\,e^{f(t)}}.
\]
Note that $q_{w,f}$ extends to $0$ as a $C^1$ function with $q_{w,f}(0)=0$ and $q'_{w,f}(0)=\frac{1}{m}$.

For a fixed number $R>0$, we define the $C^2$ function $\phi_R:[0,\infty[\to\rr$ by
\begin{equation}
\label{eq:function}
\phi_{R}(s):=\int_s^R q_{w,f}(t)\,dt.
\end{equation}
It is easy to check that
\begin{equation}
\label{poisson}
\begin{cases}
\phi_R''(s)+\phi_R'(s)\left((m-1)\,\frac{w'(s)}{w(s)}+\theta(s)\right)=-1 \ \text{ in } [0,\infty[,
\\
\phi_R(R)=0.
\end{cases}
\end{equation}

If we consider $v:=\phi_{R}\circ r$, then we obtain a radial function $v\in C^2(M-\text{cut}\{o\})$. Moreover, since $\phi_{R}'(s)\leq 0$, we infer from inequality (\ref{lapF2}) in Theorem \ref{th:laplaceineqr} that
\begin{equation}
\label{eq:strong}
\Delta^hv \geq -1 \ \text{ on } M-\text{cut}(o).
\end{equation}
Let us see that this implies
\begin{equation}
\label{eq:weak}
\int_M\escpr{\nabla v,\nabla\var}\,dv_h\leq\int_M\var\,dv_h, \ \text{ for any } \var\in H^1_0(M,dv_h) \ \text{with } \var\geq 0.
\end{equation}
For that we follow the approximation argument in \cite[Lem.~2.5]{prs-book}. Let $\{M_n\}_{n\in\nn}$ be an exhaustion of $M-\text{cut}(0)$ by smooth precompact (and nested) open sets such that $o\in M_n$ and the radial derivative with respect to the outer conormal $\nu_n$ along $\ptl M_n$ satisfies $\escpr{\nabla r,\nu_n}>0$. 
For a function $\var\in C^\infty_0(M)$ with $\var\geq 0$, equation \eqref{eq:strong} implies
\[
\int_{M_n}(\Delta^hv)\,\var\,dv_h\geq-\int_{M_n}\var\,dv_h, \ \text{ for any }n\in\nn.
\]
By taking into account that $\divv^h(\var\,\nabla v)=\var\,\Delta^hv+\escpr{\nabla v,\nabla\var}$, and applying the divergence theorem as in Lemma 2.1 of \cite{homostable}, we obtain
\[
\int_{M_n}(\Delta^hv)\,\var\,dv_h=\int_{\ptl M_n}\var\,\escpr{\nabla v,\nu_n}\,da_h-\int_{M_n}\escpr{\nabla v,\nabla\var}\,dv_h.
\]
The first integral at the right hand side is nonpositive since $\escpr{\nabla v,\nu_n}=\phi'_R(r)\,\escpr{\nabla r,\nu_n}$ along $\ptl M_n$. This shows that
\[
-\int_{M_n}\escpr{\nabla v,\nabla\var}\,dv_h\geq-\int_{M_n}\var\,dv_h, \ \text{ for any }n\in\nn.
\]
By passing to the limit we get \eqref{eq:weak} for $\var\in C^\infty_0(M)$ by using the dominated convergence theorem and the fact that $\text{cut}(o)$ has null weighted volume. By standard approximation the inequality also holds for any $\var\in H^1_0(M,dv_h)$ with $\var\geq 0$.

Next, for any $\eps>0$ small enough, we define the function $\var_\eps:=\rho_\eps\circ r$, where
\[
\rho_\eps(t):=
\begin{cases}
\frac{t}{\eps}\,\,\,&\text{if $0\leq t\leq\eps$},
\\
1\,\,\,&\text{if $\eps\leq t\leq R-\eps$}, 
\\
\frac{R-t}{\eps}\,\,\,&\text{if $R-\eps\leq t\leq R$},
\\
0\,\,\,&\text{if $t\geq R$}.
\end{cases}
\]
Clearly $\var_\eps\in H^1_0(M,dv_h)$ with $\var_\eps\geq 0$. Inequality \eqref{eq:weak} and some computations yield
\[
\frac{1}{\eps}\int_{B_R-B_{R-\eps}}q_{w,f}(r)\,dv_h-\frac{1}{\eps}\int_{B_\eps}q_{w,f}(r)\,dv_h\leq\int_{B_R}\var_\eps\,dv_h, \ \text{ for any }\eps>0.
\]
This inequality can be written as
\[
\frac{\eta(R)-\eta(R-\eps)}{\eps}-\frac{\eta(\eps)}{\eps}\leq\int_{B_R}\var_\eps\,dv_h, \ \text{ for any }\eps>0,
\]
where
\[
\eta(s):=\int_{B_s}q_{w,f}(r)\,dv_h=\int_0^s q_{w,f}(t)\Vol_h(\ptl B_t)\,dt,
\]
and we have used the coarea formula. Note that $\eta'(0)=0$ and $\eta'(s)=q_{w,f}(s)\Vol_h(\ptl B_s)$ for almost any $s>0$. By taking limits above when $\eps\to 0^+$ we conclude that inequality
\[
\Vol_h(B_R)\geq q_{w,f}(R)\Vol_h(\ptl B_R)=\frac{\Vol_f(B^w_R)}{\Vol_f(\ptl B^w_R)}\,\Vol_h(\ptl B_R)
\]
holds for almost any $R>0$. This proves \eqref{ineq:quotient}.

Now, consider the function
\[
F(R):=\frac{\Vol_h(B_R)}{\Vol_{f}(B^w_R)}, \ \ \text{ for any } R>0.
\]
Taking into account that $\frac{d}{dR} \Vol_h(B_R)=\Vol_h
(\partial B_R)$ for almost any $R>0$, that $\frac{d}{dR} \Vol_{f}(B^w_R)=\Vol_{f}(\partial B^w_R)$ for any $R>0$, and inequality \eqref{ineq:quotient}, we easily deduce that $F(R)$ is non-increasing, and so
\[
F(R)\leq \lim_{R\rightarrow 0^+} F(R)=e^{h(o)}, \ \text{ for any }R>0,
\]
which proves \eqref{ineq:volume}. The last equality above comes from the asymptotic expansion of weighted volumes for small geodesic balls \cite[Ch.~3]{bayle-thesis}, which implies that
\[
\Vol_h(B_t)\sim c_m\,t^m\, e^{h(o)} \ \text{ and } \  \Vol_f(B^w_t)\sim c_m\,t^m,
\]
where $c_m$ is a positive dimensional constant and we have used that $f(o_w)=0$. 
Finally, from \eqref{ineq:quotient} and \eqref{ineq:volume}, we conclude that
\begin{equation*}
\Vol_h(\partial B_R)\leq \frac{\Vol_h(B_R)}{\Vol_{f}(B^w_R)}\,\Vol_{f} (\partial
B^w_R)\leq  e^{h(o)}\,\Vol_{f} (\partial B^w_R),
\end{equation*}
for almost any $R>0$. This proves \eqref{ineq:volume2} and completes the proof.
\end{proof}

\begin{remark}
\label{re:paluego}
When the point $o\in M$ is a pole then inequalities \eqref{ineq:quotient} and \eqref{ineq:volume2} holds for any $R>0$. In this case $\text{cut}(o)=\emptyset$ and the proof of \eqref{ineq:quotient} is easier. Indeed, by integrating in \eqref{eq:strong} and applying the divergence theorem, it follows that
\begin{align*}
\Vol_h(B_R)&\geq -\int_{B_R}
\Delta^h v\,dv_h=-\int_{B_R}\divv^h\nabla v\,dv_h
\\
&=-\int_{\partial B_R} \langle\nabla v,\nabla r\rangle\,d
a_h=q_{w,f}(R)\,\Vol_h(\partial B_R),
\end{align*}
which proves \eqref{ineq:quotient} for any $R>0$. From here we deduce \eqref{ineq:volume2} as in the general case.
\end{remark}

In the next result we show a capacity comparison and a parabolicity criterion for weighted manifolds under a lower bound on $\text{Ric}^h_\infty$.

\begin{theorem}
\label{comp:inftycapacity} 
Let $(M^m,g,e^h)$ be a complete and non-compact weighted manifold, $r:M\rightarrow [0,\infty[$ the distance function from a point $o\in M$, and $w(s)$ a smooth function such that $w(0)=0$, $w'(0)=1$ and $w(s)>0$ for all $s>0$.  Suppose that the following conditions are fulfilled:
\begin{itemize}
\item [a) ] Every radial $\infty$-Bakry-\'Emery Ricci curvature is bounded as
\[
{\rm Ric}^h_\infty(\nabla r,\nabla r) \geq -(m-1)\,\frac{w''(r)}{w(r)} \ \text{ on }M-(\rm{cut}(o)\cup\{o\}).
\]
\item [b) ] There exists a non-decreasing $C^1$ function $\theta:[0,\infty[\to\rr$ such that  
$$\langle\nabla h,\nabla r\rangle\,w'(r)\leq \theta(r)\,w'(r) \ \text{ on }M-(\rm{cut}(o)\cup\{o\}).$$ 
\end{itemize}
Then, for almost any $\rho>0$, we have  
\begin{equation}
\label{isopCap1}
\frac{\C^h (B_\rho)}{{\rm Vol}_h(\partial B_\rho)}\leq \frac{\C^f(B^w_\rho)}{{\rm Vol}_f(\partial B^w_\rho)}.
\end{equation}
As a consequence
\begin{equation}
\label{Capcomp}
\C^h(B_\rho) \leq e^{h(o)}\,\C^{f} (B^w_\rho),
\end{equation}
where $\C^{f} (B^w_\rho)$ denotes the weighted capacity of the metric ball $B^w_\rho$ in the weighted $(w,f)$-model space $(M^m_w,g_w,e^{f(r)})$ with $f(r):=\int_0^r \theta (s)\,ds$.

Moreover, if
\begin{equation}
\label{eq:integral}
\int_\rho^\infty w^{1-m} (s)\, e^{-f(s)}\,ds=\infty,
\end{equation}
for some $\rho>0$, then $M$ is $h$-parabolic.
\end{theorem}

\begin{proof}
We first prove inequality (\ref{isopCap1}). For any numbers $\rho,R>0$ with $\rho<R$, let $\phi_{\rho,R,f}:[0,\infty[\to\rr$ be the function defined in \eqref{fsolmodel}. This provides the solution to the problem \eqref{eqfDirModel2} in the $(w,f)$-model space $(M^m_{w},g_w,e^{f(r)})$. In particular, equality \eqref{eq:sabika} is satisfied on $(0,\infty)$. Compose this function with the distance $r$ to obtain a radial function $v:=\phi_{\rho,R,f}\circ r$. It is clear that $v\in C^2(M-\text{cut}(o))$ such that $v=1$ in $\ptl B_\rho$, $v=0$ in $\ptl B_R$ and $0\leq v\leq 1$ in the annulus $A_{\rho,R}:=B_R-\overline{B}_\rho$. Since $\phi_{\rho,R,f}'\leq 0$, we can use inequality (\ref{lapF2}) in Theorem \ref{th:laplaceineqr} to deduce
\[
\Delta^h v \geq\phi_{\rho,R,f}''(r)+\phi_{\rho,R,f}'(r)\left((m-1)\,\frac{w'(r)}{w(r)}+\theta(r)\right)=0 \ \text{ on }M-\rm{cut}(o).
\]
From here, we can reproduce the approximation argument in the proof of Theorem~\ref{th:isoper_infty} to infer that
\[
\int_M\escpr{\nabla v,\nabla\var}\,dv_h\leq 0, \ \text{ for any } \var\in H^1_0(M,dv_h) \ \text{with } \var\geq 0.
\]
In particular, this inequality holds when $\var:=\psi\,v$, for any function $\psi\in H^1_0(M,dv_h)$ such that $\psi\geq 0$ and $\psi=0$ in $M-A_{\rho,R}$. This implies that
\begin{equation}
\label{eq:weak2}
\int_{M}\psi\,|\nabla v|^2\,dv_h\leq-\int_{M}v\,\escpr{\nabla v,\nabla\psi}\,dv_h.
\end{equation}

Now, we define a function $\overline{v}\in H^1_0(B_R,dv_h)$ by
\[
\overline{v}:=
\begin{cases}
v\,\,\,&\text{in $A_{\rho,R}$},
\\
1\,\,\,&\text{in $B_\rho$}. 
\end{cases}
\]
From the definition of weighted capacity, it follows that
\begin{equation}
\label{eq:fusila}
\C^h(B_\rho,B_R)\leq\int_{B_R}|\nabla\overline{v}|^2\,dv_h=\int_{A_{\rho,R}}|\nabla v|^2\,dv_h.
\end{equation}
Next, we will estimate the last integral with the help of \eqref{eq:weak2}.

For any $\eps>0$ small enough, we take the function $\psi_\eps:=\rho_\eps\circ r$, where
\[
\rho_\eps(t):=
\begin{cases}
0\,\,\,&\text{if $0\leq t\leq\rho$},
\\
\frac{t-\rho}{\eps}\,\,\,&\text{if $\rho\leq t\leq\rho+\eps$},
\\
1\,\,\,&\text{if $\rho+\eps\leq t\leq R-\eps$}, 
\\
\frac{R-t}{\eps}\,\,\,&\text{if $R-\eps\leq t\leq R$},
\\
0\,\,\,&\text{if $t\geq R$}.
\end{cases}
\]
It is clear that $\psi_\eps\in H^1_0(M,dv_h)$ with $\psi_\eps\geq 0$ and $\psi_\eps=0$ on $M-A_{\rho,R}$. Hence, inequality \eqref{eq:weak2} and some computations lead to
\begin{align*}
\int_{M}\psi_\eps\,|\nabla v|^2\,dv_h&\leq
\frac{1}{\eps}\int_{B_R-B_{R-\eps}}\phi_{\rho,R,f}(r)\,\phi_{\rho,R,f}'(r)\,dv_h
\\
&-\frac{1}{\eps}\int_{B_{\rho+\eps}-B_\rho}\phi_{\rho,R,f}(r)\,\phi_{\rho,R,f}'(r)\,dv_h, \ \text{ for any }\eps>0.
\end{align*}
This inequality can be written as
\[
\int_{M}\psi_\eps\,|\nabla v|^2\,dv_h\leq\frac{\eta(R)-\eta(R-\eps)}{\eps}-\frac{\eta(\rho+\eps)-\eta(\rho)}{\eps}, \ \text{ for any }\eps>0,
\]
where $\eta:[0,\infty[\to\rr$ is given by
\[
\eta(s):=\int_{B_s}\phi_{\rho,R,f}(r)\,\phi_{\rho,R,f}'(r)\,dv_h=\int_0^s \phi_{\rho,R,f}(t)\,\phi_{\rho,R,f}'(t)\Vol_h(\ptl B_t)\,dt,
\]
and we have used the coarea formula. By taking limits above when $\eps\to 0^+$, and having in mind the equalities $\phi_{\rho,R,f}(\rho)=1$, $\phi_{\rho,R,f}(R)=0$ together with equations \eqref{eq:fusila} and \eqref{fcapacitymodel}, we conclude that inequality
\begin{align*}
\C^h(B_\rho,B_R)&\leq\int_{A_{\rho,R}}|\nabla v|^2\,dv_h\leq |\phi'_{\rho,R,f}(\rho)|\,\Vol_h(\ptl B_\rho)
\\
&=\frac{\text{Cap}^f(B^w_\rho,B^w_R)}{\Vol_f(\ptl B^w_\rho)}\,\Vol_h(\ptl B_\rho),
\end{align*}
holds for almost every $\rho,R>0$ with $\rho<R$. Thanks to \eqref{eq:sensible} the inequality \eqref{isopCap1} follows by taking limits in the previous estimate when $R$ goes to infinity. On the other hand, the comparison in \eqref{Capcomp} comes from \eqref{isopCap1} and \eqref{ineq:volume2}.

Finally, suppose that \eqref{eq:integral} is satisfied for some $\rho>0$. This implies by equation \eqref{fcapacitymodel} that $\C^f(B^w_\rho)=0$. We can admit that \eqref{isopCap1} holds for the value $\rho$, so that $\C^h(B_\rho)=0$. From Theorem~\ref{theorGrig} we conclude that $M$ is $h$-parabolic.
\end{proof}

\begin{remark}
If the point $o\in M$ is a pole, then $\text{cut}(o)=\emptyset$, and the inequalities \eqref{isopCap1} and \eqref{Capcomp} are satisfied for any $\rho>0$. Indeed, by using the divergence theorem, the inequality $\Delta^hv\geq 0$ and the equalities $v=1$ in $\ptl B_\rho$ and $v=0$ in $\ptl B_R$, we get
\begin{equation*}
\begin{split}
\C^h(B_\rho,B_R)&\leq\int_{B_R}|\nabla\overline{v}|^2\,dv_h=\int_{A_{\rho,R}}|\nabla v|^2\,dv_h
\\
&=\int_{A_{\rho,R}}\divv^h(v\,\nabla v)\,dv_h-\int_{A_{\rho,R}}v\,\Delta^h v\,dv_h
\\
&\leq\int_{\ptl A_{\rho,R}}v\,\escpr{\nabla v,\nu}\,da_h=|\phi_{\rho,R,f}'(\rho)|\,\Vol_h(\ptl B_\rho).
\end{split}
\end{equation*}
From here we can deduce \eqref{isopCap1} and \eqref{Capcomp} as in the general case.
\end{remark}

\subsection{Comparisons under bounds on the Riemannian curvatures}\label{Riemsec}\

Here we establish estimates for the isoperimetric quotient of balls together with parabolicity and hyperbolicity criteria by assuming radial bounds on some Riemannian curvatures of the ambient manifold and on the radial derivatives of the weight.

We first obtain for the isoperimetric quotient of balls the same inequality as in \eqref{ineq:quotient}  by means of different hypotheses.

\begin{theorem} 
\label{volumeRiem1} 
Let $(M^m,g,e^h)$ be a complete and non-compact weighted manifold, $r:M\rightarrow [0,\infty[$ the distance function from a point $o\in M$, and $w(s)$ a smooth function such that $w(0)=0$, $w'(0)=1$, and $w(s)>0$ for all $s>0$.  Suppose that the following conditions are fulfilled:
\begin{itemize}
\item[a) ] Every radial Ricci curvature in $(M,g)$ is bounded as
$${\rm Ric} (\nabla r,\nabla r) \geq -(m-1)\,\frac{w''(r)}{w(r)} \ \text{ on }M-(\emph{cut}(o)\cup\{o\}).$$
\item[b) ] There exists a continuous function $\theta:[0,\infty[\to\rr$ such that
$$\langle \nabla h,\nabla r\rangle \leq \theta (r) \ \text{ on }M-(\emph{cut}(o)\cup\{o\}).$$
\end{itemize}
Then, we have
\begin{equation*} 
\frac{\Vol_h(B_R)}{\Vol_h(\partial B_R)} \geq\frac{\Vol_{f}(B^w_R)}{\Vol_{f}(\partial B^w_R)}, \ \text{ for almost any }R>0,
\end{equation*}
where $\Vol_{f}$ stands for the weighted volume in the $(w,f)$-model space $(M^m_w,g_w,e^{f(r)})$ with $f(r):=\int_0^r\theta(s)\,ds$. As a consequence
\begin{align*}
\Vol_h(B_R)&\leq e^{h(o)}\,\Vol_{f} (B^w_R), \ \text{ for any }R>0,
\\
\Vol_h(\partial B_R)&\leq e^{h(o)}\,\Vol_{f} (\partial B^w_R), \ \text{ for almost any }R>0.
\end{align*}
In particular, if $\Vol_f(M_w)<\infty$, then $\Vol_h(M)<\infty$.
\end{theorem}

\begin{proof}
For any $R>0$, we consider the function $\phi_R:[0,\infty[\to\rr$ defined in \eqref{eq:function}. This provides the solution to the Poisson problem in \eqref{poisson}. Since $\phi_{R}'(r)\leq 0$, by applying the inequalities for the Laplacian of radial functions in a Riemannian manifold with Ricci curvature bounded from below \cite{GreW, Zhu, Pa2}, or by using the inequality \eqref{lapF2} with $\theta=0$ and $h=0$, we get that the radial function $v:=\phi_{R}\circ r$ satisfies
\[
\Delta v\geq\phi_{R}''(r)+\phi'_{R}(r)\,(m-1)\,\frac{w'(r)}{w(r)} \ \text{ on }M-(\text{cut}(o)\cup\{o\}).
\]
If we combine this with the hypothesis in b), then we deduce that
$$\Delta^h v \geq -1 \ \text{ on } M-\text{cut}(o).$$
From here we can proceed as in the proof of Theorem \ref{th:isoper_infty} to show the claim.
\end{proof} 

In the next result we replace the Ricci curvature with the sectional curvature so that, by reversing the hypotheses in Theorem~\ref{volumeRiem1}, we deduce the opposite comparisons. In this way, we are able to show an upper bound for the weighted isoperimetric quotient of geodesic balls, which extends a result of Markvorsen and the second author for the unweighted case~\cite[Sect.~8]{MP5}.

\begin{theorem} 
\label{volumeRiem2} 
Let $(M^m,g,e^h)$ be a complete and non-compact weighted manifold, $r:M\rightarrow [0,\infty[$ the distance function from a point $o\in M$ and $w(s)$ a smooth function such that $w(0)=0$, $w'(0)=1$, and $w(s)>0$ for all $s>0$.  Suppose that the following conditions are fulfilled:
\begin{itemize}
\item[a) ] For any $p\in M-(\emph{cut}(o)\cup\{o\})$ and any plane $\sigma_p\subseteq T_pM$ containing $(\nabla r)_p$, we have
$${\rm Sec} (\sigma_p)\leq-\,\frac{w''(r)}{w(r)}.$$
\item[b) ] There exists a continuous function $\theta:[0,\infty[\to\rr$ such that
$$\langle \nabla h,\nabla r\rangle \geq \theta (r) \ \text{ on }M-(\emph{cut}(o)\cup\{o\}).$$
\end{itemize}
Then, for any $R>0$ such that $\overline{B}_R\cap\emph{cut}(o)=\emptyset$, we get
\begin{equation*} 
\frac{\Vol_h(B_R)}{\Vol_h(\partial B_R)} \leq
\frac{\Vol_{f}(B^w_R)}{\Vol_{f}(\partial B^w_R)},
\end{equation*}
and therefore
\begin{align*}
\Vol_h(B_R)&\geq e^{h(o)}\,\Vol_{f} (B^w_R),
\\
\Vol_h(\partial B_R)&\geq e^{h(o)}\,\Vol_{f} (\partial B^w_R),
\end{align*}
where $\Vol_{f}$ stands for the weighted volume in the $(w,f)$-model space $(M^m_w,g_w,e^{f(r)})$ with $f(r):=\int_0^r \theta(s)\,ds$.
\end{theorem}

\begin{proof}
It is well known, see for instance \cite[Ch.~2]{GreW} and \cite{Pa2}, that the upper bound on the radial sectional curvatures implies that
\begin{equation}
\label{eq:sometimes}
\Delta r\geq (m-1)\,\frac{w'(r)}{w(r)} \ \text{ on } M-(\text{cut}(o)\cup\{o\}).
\end{equation}
Take $R>0$ and let $\phi_{R}:[0,\infty[\to\rr$ be the function in \eqref{eq:function}, which provides the solution to the Poisson problem in \eqref{poisson}. We define $v:=\phi_{R}\circ r$. Since $\phi_{R}'(r)\leq 0$, from \eqref{eq:sometimes} we get
\[
\Delta v\leq\phi''_{R}(r)+\phi'_{R}(r)\,(m-1)\,\frac{w'(r)}{w(r)} \ \text{ on } M-(\text{cut}(o)\cup\{o\}).
\]
This inequality together with the hypothesis in b) yields
$$\Delta^h v \leq -1 \ \text{ on } M-\text{cut}(o).$$
Hence, if $R$ satisfies that $\overline{B}_R\cap\text{cut}(o)=\emptyset$, then $\ptl B_R$ is a smooth hypersurface, and we can proceed as in Remark~\ref{re:paluego} with reversed inequalities to deduce all the comparisons in the claim.
\end{proof} 

\begin{remark}
\label{re:technical}
As a difference with respect to Theorem~\ref{volumeRiem1}, where we assumed a lower bound on the Ricci curvature, the comparisons in Theorem~\ref{volumeRiem2} are only valid for balls having empty intersection with $\text{cut}(o)$. The reason is that the approximation argument in the proof of Theorem~\ref{th:isoper_infty} fails, so that we cannot deduce an integral inequality as in \eqref{eq:weak} from the Laplacian inequality $\Delta^h v\leq -1$ on $M-\text{cut}(o)$. This phenomenon is already observed in some classical comparison results in Riemannian geometry, see for instance \cite[Sect.~III.4]{chavel}.
\end{remark}

Now, we prove a capacity comparison with an associated parabolicity criterion under a lower bound on the radial Ricci curvatures. In the unweighted case $h=0$ we recover a result of Ichihara~\cite{I1}.

\begin{theorem}
\label{parabolicity-h''} 
Let $(M^m,g,e^h)$ be a complete and non-compact weighted manifold, $r:M\rightarrow
[0,\infty[$ the distance function from a point $o\in M$, and $w(s)$ a smooth function such that $w(0)=0$, $w'(0)=1$ and $w(s)>0$ for all $s>0$.  Suppose that the following conditions are fulfilled:
\begin{itemize}
\item[a) ] Every radial Ricci curvature in $(M,g)$ is bounded as
$${\rm Ric} (\nabla r,\nabla r) \geq -(m-1)\,\frac{w''(r)}{w(r)} \ \text{ on }M-(\emph{cut}(o)\cup\{o\}).$$
\item[b) ] There exist $\rho_0>0$ and a continuous function $\theta:[\rho_0,\infty[\to\rr$ such that
$$\langle \nabla h,\nabla r\rangle \leq \theta (r) \ \text{ on }M-(\emph{cut}(o)\cup B_{\rho_0}).$$
\end{itemize}
Then, for almost any $\rho\geq\rho_0$, we have
\begin{equation*}
\frac{\C^h (B_\rho)}{{\rm Vol}_h(\partial B_\rho)}\leq\frac{\C^f(B^w_\rho)}{\Vol_f(\ptl B^w_\rho)},
\end{equation*}
where $\C^{f} (B^w_\rho)$ denotes the weighted capacity of the metric ball $B^w_\rho$ in a weighted $(w,f)$-model space $(M^m_w,g_w,e^{f(r)})$ with $f(r):=\int_{\rho_0}^r \theta (s)\,ds$ for any $r\geq\rho_0$. 

Moreover, if $\rho_0=0$, then
\begin{equation*}
\label{Capcomp1}
\C^h(B_\rho) \leq e^{h(o)}\,\C^{f} (B^w_\rho),
\end{equation*}
for almost any $\rho>0$.

Anyway, if
\begin{equation}
\label{cond-para-h''}
\int_{\rho_0}^\infty w^{1-m}(s)\, e^{-f(s)}\,ds=\infty,
\end{equation}
then $M$ is $h$-parabolic.
\end{theorem}

\begin{proof} 
We follow the scheme of proof and notation in Theorem~\ref{comp:inftycapacity}. For any numbers $\rho,R$ with $\rho_0\leq\rho<R$, consider the radial function $v:=\phi_{\rho,R,f}\circ r$. Since $\phi_{\rho,R,f}'\leq 0$, by taking into account the estimate for $\Delta v$ in \eqref{lapF2} and the inequality $\langle\nabla h,\nabla r\rangle\leq\theta(r)$ on $M-(\text{cut}(o)\cup B_{\rho_0})$, we obtain
\begin{align*}
\Delta^h v&=\Delta v+\escpr{\nabla h,\nabla v}\geq\phi_{\rho,R,f}''(r)+\phi_{\rho,R,f}'(r)\left((m-1)\,\frac{w'(r)}{w(r)}+\escpr{\nabla h,\nabla r}\right)
\\
&\geq\phi_{\rho,R,f}''(r)+\phi_{\rho,R,f}'(r)\left((m-1)\,\frac{w'(r)}{w(r)}+\theta(r)\right)=0 \ \text { on } M-(\text{cut}(o)\cup B_{\rho_{0}}).
\end{align*}
By reasoning as in the proof of Theorem~\ref{th:isoper_infty}, we get
\[
\int_M\escpr{\nabla v,\nabla\var}\,dv_h\leq 0, \ \text{ for any }\var\in H^1_0(M-\overline{B}_{\rho_0},dv_h) \text{ with } \var\geq 0.
\]
In particular, the inequality in \eqref{eq:weak2} is valid for any $\psi\in H^1_0(M-\overline{B}_{\rho_0},dv_h)$ with $\psi\geq 0$ and $\psi=0$ in $A_{\rho,R}:=B_R-\overline{B}_\rho$. From this point we can finish the proof as in Theorem~\ref{comp:inftycapacity} with the help of the last inequality in the statement of Theorem~\ref{volumeRiem1}.
\end{proof}

If we replace the Ricci curvature with the sectional curvature, and we reverse the hypotheses in Theorem~\ref{parabolicity-h''}, then we deduce a weighted hyperbolicity criterion for complete manifolds with a pole which generalizes a result of Ichihara \cite{I1} for simply connected Riemannian manifolds. Indeed, under the assumptions on the sectional curvature, the fact that $M$ is simply connected implies that it has a pole. We also remark that the technical issue observed in Remark~\ref{re:technical} prevents a direct extension of the theorem to any complete weighted manifold.

\begin{theorem}
\label{hyperbolicity-h''}
Let $(M^m,g,e^h)$ be a complete weighted manifold, $r:M\rightarrow[0,\infty[$ the distance function from a pole $o\in M$, and $w(s)$ a smooth function such that $w(0)=0$, $w'(0)=1$ and $w(s)>0$ for all $s>0$.  Suppose that the following conditions are fulfilled:
\begin{itemize}
\item[a) ] For any $p\in M-\{o\}$ and any plane $\sigma_p\subeq T_pM$ containing $(\nabla r)_p$, we have
$${\rm Sec}(\sigma_p)\leq-\frac{w''(r)}{w(r)}.$$
\item[b) ] There exist $\rho_0>0$ and a continuous function $\theta:[\rho_0,\infty[\to\rr$ such that
$$\langle \nabla h,\nabla r\rangle \geq \theta (r) \ \text{ on }M-B_{\rho_0}.$$
\end{itemize}
Then, for any $\rho\geq\rho_0$, we get
\begin{equation}
\label{quasineq-hyperbolicity-h''}
\frac{\C^h (B_\rho)}{{\rm Vol}_h(\partial B_\rho)}\geq\frac{\C^f(B^w_\rho)}{\Vol_f(\ptl B^w_\rho)},
\end{equation}
where $\C^{f} (B^w_\rho)$ denotes the weighted capacity of the metric ball $B^w_\rho$ in a weighted $(w,f)$-model space $(M^m_w,g_w,e^{f(r)})$ with $f(r):=\int_{\rho_0}^r \theta (s)\,ds$ for any $r\geq\rho_0$.

Moreover, if $\rho_0=0$, then
\begin{equation}
\label{Capcomp2}
\C^h(B_\rho)\geq e^{h(o)}\,\C^{f} (B^w_\rho),
\end{equation}
for any $\rho>0$.

Anyway, if
\begin{equation}
\label{cond-hyper-h''}
\int_{\rho_0}^\infty w^{1-m}(s)\, e^{-f(s)}\,ds<\infty,
\end{equation}
then $M$ is $h$-hyperbolic.
\end{theorem}

\begin{proof} 
We follow the notation in Theorem~\ref{comp:inftycapacity}. Choose numbers $\rho,R$ with $\rho_0\leq\rho<R$. Consider the radial function $v:=\phi_{\rho,R,f}\circ r$ defined in $A_{\rho,R}:=B_R-\overline{B}_\rho$. 
Since $\phi_{\rho,R,f}'\leq 0$, by taking into account the estimate for $\Delta r$ in \eqref{eq:sometimes}, the equality \eqref{eq:bullin} and the hypothesis $\langle\nabla h,\nabla r\rangle\geq\theta(r)$ on $M-B_{\rho_0}$, we obtain
\begin{equation*}
\Delta^h v=\Delta v+\escpr{\nabla h,\nabla v}\leq\phi_{\rho,R,f}''(r)+\phi_{\rho,R,f}'(r)\left((m-1)\,\frac{w'(r)}{w(r)}+\theta(r)\right)=0.
\end{equation*}

On the other hand, the $h$-capacity potential $u$ of the capacitor $(\overline{B}_\rho,B_R)$ in $M$ satisfies $\Delta^hu=0$ and $0\leq u\leq 1$ in $A_{\rho,R}$, whereas $u=v=1$ along $\ptl B_\rho$ and $u=v=0$ along $\ptl B_R$. From equation \eqref{eq:capint}, and applying the divergence theorem in $A_{\rho,R}$ to the vector fields $v\,\nabla u$ and $u\,\nabla v$, we get
\begin{align*}
\C^h(B_\rho,B_R)&=-\int_{\ptl B_\rho}\escpr{\nabla u,\nabla r}\,da_h=\int_{A_{\rho,R}}\escpr{\nabla v,\nabla u}\,dv_h
\\
&\geq\int_{A_{\rho,R}}u\,\Delta^h v\,dv_h+\int_{A_{\rho,R}}\escpr{\nabla u,\nabla v}\,dv_h
\\
&=-\int_{\ptl B_\rho}\escpr{\nabla v,\nabla r}\,da_h= |\phi_{\rho,R,f}'(\rho)|\,\Vol_h(\ptl B_\rho)
\\
&=\frac{\text{Cap}^f(B^w_\rho,B^w_R)}{\Vol_f(\ptl B^w_\rho)}\,\Vol_h(\ptl B_\rho),
\end{align*}
where we have used \eqref{fcapacitymodel}. Thanks to \eqref{eq:sensible} the inequality \eqref{quasineq-hyperbolicity-h''} follows by taking limits above when $R$ goes to infinity. The estimate in \eqref{Capcomp2} comes from \eqref{quasineq-hyperbolicity-h''} with the help of the last comparison in Theorem~\ref{volumeRiem2}. Finally, the hypothesis \eqref{cond-hyper-h''} implies by equation \eqref{fcapacitymodel} that $\C^f(B^w_{\rho_0})>0$. Thus, we have $\C^h(B_{\rho_0})>0$ by \eqref{quasineq-hyperbolicity-h''}. From Theorem~\ref{theorGrig} we conclude that $M$ is $h$-hyperbolic.
\end{proof}

\begin{remark} 
\label{re:papa}
The hypothesis about $\text{Ric}$ in Theorem~\ref{parabolicity-h''} is independent of the hypothesis about $\text{Ric}^h_\infty$ in Theorem~\ref{comp:inftycapacity}. However, it is immediate that the second one implies the first one provided the convexity condition $(\text{Hess}\,h)(\nabla r,\nabla r)\geq 0$ holds on $M-(\text{cut}(o)\cup\{o\})$. In a similar way, the hypothesis
\[
\text{Sec}^h_\infty(\sigma_p)\leq-\frac{w''(r)}{w(r)}
\]
for any $p\in M-(\text{cut}(o)\cup\{o\})$ and any radial plane $\sigma_p\subseteq T_pM$ entails the same upper bound on $\text{Sec}(\sg_p)$ whenever $(\text{Hess}\,h)(\nabla r,\nabla r)\leq 0$ on $M-(\text{cut}(o)\cup\{o\})$.
\end{remark}

\subsection{Comparisons under a lower bound on the $q$-Bakry-\'Emery 
Ricci curvatures} 
\label{finiteBakry}\

In this subsection we first follow the arguments of the previous ones to provide a comparison for the weighted isoperimetric quotient and a parabolicity criterion by assuming a lower bound on some $q$-Bakry-\'Emery Ricci curvature with $q>0$. Related comparisons for volumes and quotient of volumes (but not for isoperimetric quotients) are found in \cite{qian,MR2243959,MRS}.

\begin{theorem}
\label{th:isoper} 
Let $(M^m,g,e^h)$ be a complete and non-compact weighted manifold, $r:M\rightarrow [0,\infty[$ the distance function from a point $o\in M$, and $w(s)$ a smooth function such that $w(0)=0$, $w'(0)=1$ and $w(s)>0$ for all $s>0$. 

If, for some $q>0$, the radial $q$-Bakry-\'Emery Ricci curvature is bounded as
\[
{\rm Ric}^h_q(\nabla r,\nabla r)\geq-(m+q-1)\,\frac{w''(r)}{w(r)} \ \text{ on } M-(\emph{cut}(o)\cup\{o\}),
\]
then, for almost any $R>0$, we have
\[
\frac{\Vol_h(B_R)}{\Vol_h (\partial B_R)}\geq\frac{\int_0^R w^{m+q-1}(s)\,ds}{w^{m+q-1}(R)}.
\]
\end{theorem}

\begin{proof}
We will employ arguments similar to those in \cite{Pa,MaP1,MP5}. Let $q_w:(0,\infty)\to\rr$ be the function
\[
q_w(t):=\frac{\int_0^tw^{m+q-1}(s)\,ds}{w^{m+q-1}(t)}.
\]
Note that $q_{w}$ extends to $0$ as a $C^1$ function with $q_{w}(0)=0$ and $q'_{w}(0)=\frac{1}{m+q}$. 

For fixed $R>0$, we define the $C^2$ function $\phi_R:[0,\infty[\to\rr$ by $\phi_R(s):=\int_s^R q_w(t)\,dt$. This satisfies the Poisson-type problem
\begin{equation*}
\begin{cases}
\phi_R''(s)+\phi_R'(s)\,(m+q-1)\,\frac{w'(s)}{w(s)}=-1 \ \text{ in } [0,\infty[,
\\
\phi_R(R)=0.
\end{cases}
\end{equation*}
If we consider $v:=\phi_{R}\circ r$, then we get a radial function $v\in C^2(M-\text{cut}(o))$. Since $\phi_R'(s)=-q_w(s)\leq 0$, by inequality (\ref{lapF1}) in Theorem \ref{th:laplaceineqr} we deduce
\begin{displaymath}
\Delta^h v\geq \phi_R''(r)+\phi_R'(r)\,(m+q-1)\,\frac{w'(r)}{w(r)}=-1 \ \text{ on } M-\text{cut}(o).
\end{displaymath}
Now, we can follow the proof of \eqref{ineq:quotient} to deduce the desired comparison.
\end{proof}

\begin{theorem}
\label{comp:qcapacity}
 Let $(M^m,g,e^h)$ be a complete and non-compact weighted manifold, $r:M\rightarrow[0,\infty[$ the distance function from a point $o\in M$, and $w(s)$ a smooth function such that $w(0)=0$, $w'(0)=1$ and $w(s)>0$ for all $s>0$.

If, for some $q>0$, the radial $q$-Bakry-\'Emery Ricci curvature is bounded as
\[
{\rm Ric}^h_q(\nabla r,\nabla r)\geq-(m+q-1)\,\frac{w''(r)}{w(r)} \ \text{ on } M-(\emph{cut}(o)\cup\{o\}),
\]
then, for almost any $\rho>0$, we have 
\begin{displaymath}
\frac{\C^h (B_\rho)}{\Vol_h(\ptl B_\rho)}\leq\frac{w^{1-m-q}(\rho)}{\int_\rho^\infty w^{1-m-q}(s)\,ds}.
\end{displaymath}

Moreover, if
\begin{equation}
\label{cond_para}
\int_\rho^\infty w^{1-m-q}(s)\,ds=\infty,
\end{equation}
for some $\rho>0$, then $M$ is $h$-parabolic.
\end{theorem}

\begin{proof}
In this case we consider $v:=\phi_{\rho,R}\circ r$, where $\phi_{\rho,R}:[0,\infty[\to\rr$ is the function defined in \eqref{fsolmodel} when we take $f=0$ and dimensional constant $m+q$. In particular $\phi_{\rho,R}$ satisfies the differential equation in \eqref{eq:sabika}, so that the comparison $\Delta^h v\geq 0$ on $M-\text{cut}(o)$ comes from inequality \eqref{lapF1} in Theorem \ref{th:laplaceineqr}. The rest of the proof relies on the arguments employed in Theorem~\ref{comp:inftycapacity}.
\end{proof}

We finish this section by showing how we can deduce the parabolicity of a Riemannian manifold $(M^m,g)$ from a lower bound for the $q$-Bakry-\'Emery Ricci curvatures associated to a weight $e^h$ in $M$. The strategy for the proof is similar to previous ones by using a second order differential operator $L$, that coincides with the weighted Laplacian in some $(w,f)$-model space only when $q\in\mathbb{N}$.

\begin{theorem}
\label{comp:capacity} 
Let $(M^m,g)$ be a complete and non-compact Riemannian manifold, $r:M\rightarrow[0,\infty[$ the distance function from a point $o\in M$, and $w(s)$ a smooth function such that $w(0)=0$, $w'(0)=1$ and $w(s)>0$ for all $s>0$.  Suppose that there exists a weight $e^h$ on $M$ such that the following conditions are fulfilled:
\begin{itemize}
\item [a) ] For some $q>0$ the radial $q$-Bakry-\'Emery Ricci
curvature is bounded as
\[
{\rm Ric}_q^h(\nabla r,\nabla r) \geq -(m+q-1)\,\frac{w''(r)}{w(r)} \ \text{ on }M-(\emph{cut}(o)\cup\{o\}).
\]
\item [b) ] There exist $\rho_0>0$ and a continuous function $\theta:[\rho_0,\infty[\to\rr$ such that  
$$\langle\nabla h,\nabla r\rangle\geq\theta(r) \ \text{ on }M-B_{\rho_0}.$$ 
\end{itemize}
Then, for almost any $\rho\geq\rho_0$, we have  
\begin{equation}
\label{memento}
\frac{\C(B_\rho)}{{\rm Vol}(\partial B_\rho)}\leq\frac{w^{1-m-q}(\rho)\,e^{-f(\rho)}}{\int_\rho^\infty w^{1-m-q}(s)\,e^{-f(s)}\,ds},
\end{equation}
where $\C$ and $\Vol$ denote the capacity and volume in $(M,g)$, and $f(r):=-\int_{\rho_0}^r \theta (s)\,ds$ for any $r\geq\rho_0$.

Moreover, if
\begin{equation}
\label{cond_para_Riem}
\int_{\rho_0}^\infty w^{1-m-q} (s)\, e^{-f(s)}\,ds=\infty,
\end{equation}
then $M$ is parabolic in Riemannian sense.
\end{theorem}

\begin{proof}
We define a second order differential operator $L$ acting on any $C^2$ function $F:[\rho_0,\infty[\to\rr$ by 
\begin{eqnarray}
\label{ineq_Riem_Laplace}
L(F(s)):= F''(s)+F'(s)\left((m+q-1)\,\frac{w'(s)}{w(s)}-\theta(s)\right).
\end{eqnarray}
As in Theorem~\ref{capacityweighted} it is easy to check that, for any $\rho,R$ with $\rho_0\leq\rho<R$, the function $\phi^L_{\rho,R}:[\rho_0,\infty[\to\rr$ obtained by replacing $m$ with $m+q$ in \eqref{fsolmodel} satisfies equation $L(\phi^L_{\rho,R}(s))=0$ in $[\rho_0,\infty[$ with $\phi^L_{\rho,R}(\rho)=1$ and $\phi^L_{\rho,R}(R)=0$.

In the set $M-B_{\rho_0}$ we define $v:=\phi^L_{\rho,R}\circ r$, which is $C^2$ function on $M-(\text{cut}(o)\cup B_{\rho_0})$. Because $(\phi^L_{\rho,R})'\leq 0$, we can use inequality \eqref{lapF1} in Theorem \ref{th:laplaceineqr} to deduce 
\begin{equation*}
\Delta^h v\geq (\phi^L_{\rho,R})''(r)+(\phi^L_{\rho,R})'(r)\,(m+q-1)\,
\frac{w'(r)}{w(r)} \ \text{ on } M-(\text{cut}(o)\cup B_{\rho_0}).
\end{equation*}
Since $\Delta v=\Delta^h v-\langle \nabla h, \nabla v\rangle$, by taking into account the hypothesis b) and the fact that $(\phi^L_{\rho,R})'\leq 0$, we get this inequality on $M-(\text{cut}(o)\cup B_{\rho_0})$
\begin{equation*}
\begin{aligned}
\Delta v &\geq (\phi^L_{\rho,R})''(r)+(\phi^L_{\rho,R})'(r)\left((m+q-1)\,
\frac{w'(r)}{w(r)}-\langle \nabla h, \nabla r\rangle\right)
\\
&\geq (\phi^L_{\rho,R})''(r)+(\phi^L_{\rho,R})'(r)\left((m+q-1)\,
\frac{w'(r)}{w(r)}-\theta(r)\right)=0.
\end{aligned}
\end{equation*}
As in the proof of Theorem~\ref{comp:inftycapacity}, the approximation argument in Theorem~\ref{th:isoper_infty} implies
\[
\int_M\escpr{\nabla v,\nabla\var}\,dv\leq 0, \ \text{ for any }\var\in H^1_0(M-\overline{B}_{\rho_0},dv) \text{ with } \var\geq 0.
\]
In particular, the inequality in \eqref{eq:weak2} is valid for any $\psi\in H^1_0(M-\overline{B}_{\rho_0},dv_h)$ with $\psi\geq 0$ and $\psi=0$ in the annulus $A_{\rho,R}:=B_R-\overline{B}_\rho$.

From here, we can reproduce the arguments in the proof of \eqref{isopCap1} to deduce that
\begin{equation*}
\begin{aligned}
\text{Cap}(B_\rho,B_R)&\leq\int_{A_{\rho,R}}|\nabla v|^2\,da\leq|(\phi^L_{\rho,R})'(\rho)|\,\Vol_h(\ptl B_\rho)
\\
&=\frac{w^{1-m-q}(\rho)\,e^{-f(\rho)}}{\int_\rho^R w^{1-m-q}(s)\,e^{-f(s)}\,ds}\,\Vol(\partial B_\rho)
\end{aligned}
\end{equation*}
holds for almost any $\rho,R$ with $\rho_0\leq\rho<R$. 
Hence, the comparison in \eqref{memento} follows by taking limits when $R\to\infty$. Finally, the condition \eqref{cond_para_Riem} implies that $\C(B_\rho)=0$ for some $\rho\geq\rho_0$ where \eqref{memento} is valid. Thus, $M$ is parabolic by Theorem~\ref{theorGrig}.
\end{proof}

\begin{remark}
Note that, if $q\in\mathbb{N}$, then the right hand side terms in the comparisons established in Theorems~\ref{th:isoper}, \ref{comp:qcapacity} and \ref{comp:capacity} can be written as
\[
\frac{\Vol(B^w_{R,m+q})}{\Vol(\ptl B^w_{R,m+q})},  \frac{\C(B^w_{\rho,m+q})}{\Vol(\ptl B^w_{\rho,m+q})} \text{ and } \frac{\C^f(B^w_{\rho,m+q})}{\Vol_f(\ptl B^w_{\rho,m+q})} ,
\]
where $B^w_{t,m+q}$ denotes the metric ball of radius $t$ centered at $o_w$, $\Vol$ and $\C$ are the capacity and volume in the $w$-model space $(M^{m+q}_w,g_w)$, and $\Vol_f$ and $\C^f$ are the weighted capacity and volume in the $(w,f)$-model space $(M^{m+q}_w,g_w,e^{f(r)})$. In this case, the operator $L$ defined in \eqref{ineq_Riem_Laplace} coincides with the weighted Laplacian operator in $(M^{m+q}_w,g_w,e^{f(r)})$ over radial functions.

Note also that, under the conditions of Theorem~\ref{th:isoper}, we cannot deduce a volume comparison as in \eqref{ineq:volume}. The problem is that, although the corresponding function
\[
F(R):=\frac{\Vol_h(B_R)}{\Vol(B^w_{R,m+q})}
\]
is still non-increasing, we have that $F(R)\to\infty$ when $R\to 0$.
\end{remark}

\begin{remark}
\label{re:salinas}
The non-integrability (resp. integrability) hypotheses in \eqref{eq:integral}, \eqref{cond-para-h''}, \eqref{cond_para}, \eqref{cond_para_Riem} and \eqref{cond-hyper-h''} are equivalent by Remark~\ref{Alhfors} to the weighted parabolicity (resp. hyperbolicity) of the corresponding weighted comparison model. In particular, Theorems~\ref{comp:inftycapacity}, \ref{parabolicity-h''}, \ref{comp:qcapacity}, \ref{comp:capacity} and \ref{hyperbolicity-h''} show that the ambient manifold is $h$-parabolic (resp. $h$-hyperbolic) provided the weighted $(w,f)$-model space is $f$-parabolic (resp. $f$-hyperbolic).
\end{remark}

\section{Extrinsic comparison results} 
\label{extrinsic}

In this section, given a weighted manifold with a pole and a properly immersed submanifold, we establish volume and capacity comparisons for extrinsic balls in the submanifold. As a consequence, we deduce parabolicity and hyperbolicity of submanifolds by assuming certain control on the weighted mean curvature of the submanifold, the radial curvatures of the weight and some (weighted) curvatures of the ambient manifold. This extends to arbitrary weighted manifolds the results obtained by the authors in rotationally symmetric manifolds with weights \cite[Sect.~3]{HPR1}.

\subsection{Submanifolds in weighted manifolds}
\label{submanifolds}\

Let $P^n$ be an $n$-dimensional submanifold with $\ptl P=\emptyset$ properly immersed in a weighted manifold $(M^m,g,e^h)$ with a pole $o\in M$. We consider in $P$ the induced Riemannian metric. We use the notation $\nabla_P u$ and $\Delta_Pu$ for the gradient and Laplacian in $P$ of a function $u\in C^2(P)$. 

The restriction to $P$ of the weight $e^h$ in $M$ produces a structure of weighted manifold in $P$. From \eqref{eq:flaplacian} the associated  \emph{$h$-Laplacian} $\Delta^h_P$ has the expression
\begin{equation*}
\Delta^h_Pu=\Delta_Pu+\escpr{\nabla_P h,\nabla_Pu},
\end{equation*}
for any $u\in C^2(P)$. We say that the submanifold $P$ is \emph{$h$-parabolic} when $P$ is weighted parabolic as a weighted manifold. Otherwise we say that $P$ is \emph{$h$-hyperbolic}. By Theorem~\ref{theorGrig} the $h$-parabolicity of $P$ is equivalent to that $\text{Cap}^h_P(D)=0$ for some precompact open set $D\subeq P$, where $\text{Cap}^h_P$ denotes the \emph{$h$-capacity relative to $P$}.  Clearly a compact submanifold $P$ is $h$-parabolic.

Next we introduce the extrinsic balls of a submanifold $P$. As in the previous sections we denote by $r:M\to[0,\infty[$ the distance function from the pole $o\in M$, and by $B_R$ the metric ball in $M$ of radius $R>0$ centered at $o$. 

\begin{definition}
\label{extball} 
If $P$ is a non-compact submanifold properly immersed in $M$, the \emph{extrinsic metric ball} of (sufficiently large) radius $R>0$ and center $o$ is denoted by $D_R$, and defined as any connected component of the set
\begin{displaymath}
B_R\cap P=\{p\in P:\,r(p)<R\}.
\end{displaymath}
For given radii $\rho,R>0$ with $\rho<R$, we define the \emph{extrinsic annulus} in $P$ as the set
\[
A^P_{\rho,R}:=D_R-\overline{D}_\rho,
\]
where $D_R$ is the component of $B_R\cap P$ containing $D_\rho$.
\end{definition}

Since $P$ is properly immersed in $M$ the extrinsic balls are precompact open sets in $P$. As we assume that $P$ is noncompact then $D_R\neq P$ for any $R>0$. Moreover, by Sard's Theorem we deduce that $\ptl D_R$ is smooth for almost any $R>0$.

Now we present another necessary ingredient to establish our results: the weighted mean curvature of submanifolds. In the case of two-sided hypersurfaces this was first introduced by Gromov~\cite{gromov-GAFA}, see also \cite[Ch.~3]{bayle-thesis}.

\begin{definition} 
The \emph{weighted mean curvature vector} or \emph{$h$-mean curvature vector} of $P$ is the vector field normal to $P$ given by
\begin{equation*}
\label{eq:mcvector*}
\ovh^h:=n\ovh-(\nabla h)^\bot,
\end{equation*}
where $(\nabla h)^\bot$ is the normal projection of $\nabla h$ with respect to $P$ and $\ovh$ is the mean curvature vector of $P$. This is defined as $n\ovh:=-\sum_{i=1}^{m-n}\,(\divv_P N_i)\,N_i$, where $\divv_P$ stands for the divergence relative to $P$ and $\{N_1,\ldots,N_{m-n}\}$ is any local orthonormal basis of vector fields normal to $P$. 

We say that $P$ has \emph{constant $h$-mean curvature} if $|\ovh^h|$ is constant on $P$. If $\ovh^h=0$, then $P$ is called \emph{$h$-minimal}. More generally, $P$ has \emph{bounded $h$-mean curvature} if $|\ovh^h| \leq c$ on $P$ for some constant $c >0$.
\end{definition}

For later use we must note that equality
\begin{equation}
\label{relation-H-H_h}
\langle n \ovh,\nabla r \rangle +\langle \nabla_P h,\nabla_P r\rangle=
\langle \ovh^h,\nabla r \rangle +\langle \nabla h,\nabla r\rangle
\end{equation}
holds on $P-\{o\}$. This easily comes from the definition of $\ovh^h$ and the fact that $\nabla h-(\nabla h)^\bot=\nabla_Ph$.

\subsection{Extrinsic Laplacian comparisons}
\label{extLapcomp}\

As in Section~\ref{intcompan}, the analysis of modified distance functions will be instrumental to deduce our comparisons for extrinsic balls of submanifolds. The results in this subsection are applications to the extrinsic context of the estimates for the distance function in Theorem \ref{th:hessianineqr}, and of Laplacian comparisons for modified distance functions in submanifolds given in \cite{MP3,MP5,Pa2}. 

We first establish some inequalities for the weighted Laplacian of submanifolds under bounds on the radial sectional curvatures of the ambient manifold. In the particular case of rotationally symmetric manifolds with a pole it was shown in Lemma 3.1 of \cite{HPR1} that all the estimates in the next statement become equalities.

\begin{theorem}
\label{Laplacian-sub-infty}
Let $(M^m,g, e^h)$ be a weighted manifold, $P^n$ a submanifold immersed in $M$, $r:M\to[0,\infty[$ the distance function from a pole $o \in M$, and  $w(s)$ a smooth function such that $w(0)=0$ and $w(s)>0$ for all $s>0$. 

If, for any $p\in M-\{o\}$ and any plane $\sigma_p\subseteq T_pM$ containing $(\nabla r)_p$, we have
\begin{displaymath}
{\rm Sec}(\sigma_p) \geq (\leq)\,-\frac{w''(r)}{w(r)},
\end{displaymath}
then, for every smooth function $F:(0,\infty)\to\rr$ with $F'\leq 0$, we obtain the inequality
\begin{equation*}
\begin{aligned}
\Delta^h_P (F\circ r) &\geq (\leq ) \left(F''(r)-F'(r)\,\frac{w'(r)}{w(r)}\right)|\nabla_Pr|^2
\\
&+F'(r)\left(n\,\frac{w'(r)}{w(r)}+\langle \ovh^h, \nabla r\rangle +\langle \nabla h, \nabla r\rangle \right)
\end{aligned}
\end{equation*}
in the points of $P-\{o\}$.
\end{theorem}

\begin{proof}
From the results in \cite{MP5,Pa2}, the bound for the radial sectional curvatures of the ambient manifold implies that the Laplacian $\Delta_P$ of the modified distance function $F\circ r$ satisfies the inequality
\begin{displaymath}
\Delta_P(F\circ r) \geq (\leq) \left(F''(r)-F'(r)\,\frac{w'(r)}{w(r)}\right) |\nabla_P r|^2+nF'(r)\left(\frac{w'(r)}{w(r)}+\langle \ovh,\nabla r\rangle\right).
\end{displaymath}
Thus, by the definition of weighted Laplacian, we get
\begin{equation*}
\begin{aligned}
\Delta^h_P(F\circ r)&\geq (\leq) \left(F''(r)-F'(r)\,\frac{w'(r)}{w(r)}\right) |\nabla_P r|^2\\&+nF'(r)\left( \frac{w'(r)}{w(r)}+\langle \ovh,\nabla r\rangle\right)+F'(r)\,\langle \nabla_P h,\nabla_P r\rangle,
\end{aligned}
\end{equation*}
so that the claim follows by using \eqref{relation-H-H_h}.
\end{proof}

Now, we derive a comparison for the weighted Laplacian of submanifolds by assuming a lower bound on some $q$-weighted sectional curvature. Such a bound does not imply a lower bound on the Riemannian sectional curvature, so that we cannot use, as we did in Theorem~\ref{Laplacian-sub-infty}, the known inequalities for the unweighted Laplacian.

\begin{theorem}
\label{laplace-ineq-q} 
Let $(M^m,g, e^h)$ be a weighted manifold, $P^n$ a submanifold immersed in $M$, $r:M\to[0,\infty[$ the distance function from a pole $o \in M$, and  $w(s)$ a smooth function such that $w(0)=0$ and $w(s)>0$ for all $s>0$. 

If there is $q>0$ such that, for any $p\in M-\{o\}$ and any plane $\sigma_p\subseteq T_pM$ containing $(\nabla r)_p$, we have
\begin{displaymath}
{\rm Sec}_{q}^h(\sigma_p) \geq
-\,\frac{m+q-1}{m-1}\,\frac{w''(r)}{w(r)},
\end{displaymath}
then, for every smooth function $F:(0,\infty)\to\rr$ with $F'\leq 0$, we obtain the inequality 
\begin{equation*}
\begin{aligned}
\Delta^h_P (F\circ r) &\geq \left(F''(r)-\frac{F'(r)}{m-1}
\left((m+q-1)\,\frac{w'(r)}{w(r)}-\langle \nabla h,\nabla r\rangle\right)\right)|\nabla_Pr|^2
\\
&+F'(r)\left(n\,\frac{m+q-1}{m-1}\,\frac{w'(r)}{w(r)}+\frac{m-n-1}{m-1}\,
\langle \nabla h,\nabla r\rangle+\langle \ovh^h,\nabla
r\rangle\right).
\end{aligned}
\end{equation*}
in the points of $P-\{o\}$.
\end{theorem}

\begin{proof}  
For any smooth function $F:(a,b)\to\rr$ and any tangent vector $y$ to $P$, it is not difficult to check as in \cite{JK} or \cite{Pa2} that
\begin{align*}
\Hess_P(F\circ r)(y,y)&=F''(r)\,\escpr{y,\nabla r}^2
\\
&+F'(r)\,\left((\Hess r)(y,y)+\escpr{\alpha(y,y),\nabla r}\right),
\end{align*}
where $\Hess_P$ is the Hessian operator relative to $P$ and $\alpha$ is the second fundamental form of $P$. Hence, by using the estimate for $(\text{Hess}\,r)(y,y)$ in \eqref{eq:setenil} and the fact that $F'\leq 0$, we get
\begin{align*}
\Hess_P(F\circ r)(y,y) &\geq F''(r)\,\escpr{y,\nabla r}^2+F'(r)\,
\escpr{\alpha(y,y),\nabla r}
\\
&+F'(r)\,(|y|^2-\escpr{y,\nabla r}^2)\,\bigg(\frac{m+q-1}{m-1}\,\frac{w'(r)}{w(r)}-\frac{1}{m-1}\,\escpr{\nabla h,\nabla r}\bigg).
\end{align*}
Applying the previous inequality to an orthonormal basis $\{y_1,\ldots,y_n\}$ of tangent vectors to $P-\{o\}$ and summing up, we arrive at
\begin{align*}
\Delta_P(F\circ r)&\geq F''(r)\,|\nabla_Pr|^2+F'(r)\,\escpr{n\overline{H}_P,\nabla r}
\\
&+F'(r)\,\big(n-|\nabla_P r|^2\big)\,\bigg(\frac{m+q-1}{m-1}\,\frac{w'(r)}{w(r)}-\frac{1}{m-1}\,\escpr{\nabla h,\nabla r}\bigg).
\end{align*}
From here the proof finishes after some computations by taking into account \eqref{relation-H-H_h} and that $\Delta^h_P(F\circ r)=\Delta_P(F\circ r)+F'(r)\,\escpr{\nabla_P h,\nabla_P r}$.
\end{proof}

\subsection{Comparisons under bounds on the sectional curvatures.}
\label{extsect1}\

With Theorem \ref{Laplacian-sub-infty} in hand we are now ready to prove estimates for the weighted volume of extrinsic balls for submanifolds.

\begin{theorem}
\label{th:simpson}
Let $(M^m,g,e^h)$ be a weighted manifold, $P^n$ a non-compact submanifold properly immersed in $M$, $r:M\rightarrow [0,\infty[$ the distance function from a pole $o\in M$, and $w(s)$ a smooth function such that $w(0)=0$, $w'(0)=1$
and $w(s)>0$ for all $s>0$. Suppose that the following conditions are fulfilled:
\begin{itemize}
\item[(i)] For any $p\in M-\{o\}$ and any plane $\sigma_p\subseteq T_pM$ containing $(\nabla r)_p$, we have
\begin{displaymath}
{\rm Sec}(\sigma_p) \geq (\leq ) -\frac{w''(r)}{w(r)}.
\end{displaymath}
\item[(ii)] There exist continuous functions $\psi,\varphi:[0,\infty[\to\rr$, such that
\begin{displaymath}
\langle \nabla h,\nabla r\rangle\leq (\geq )\,\psi(r), \quad \langle \ovh^h,\nabla r\rangle \leq (\geq )\,\varphi(r) \ \text{ on }P-\{o\}.
\end{displaymath}
\item[(iii)] In $P-\{o\}$ the bounding functions verify
\begin{displaymath}
n\,\frac{w'(r)}{w(r)}+\psi(r)+\varphi(r)\leq (\geq )\,\frac{1}{q_{w,f}(r)}\quad {\textrm {$($balance condition$)$}},
\end{displaymath}
where $q_{w,f}$ is the weighted isoperimetric quotient in the weighted $(w,f)$-model space $(M_w^n,g_w,e^{f(r)})$ with $f(r):=\int_0^r (\psi(s)+\varphi(s))\,ds$.
\end{itemize}
Then, for any extrinsic ball $D_R$ in $P$ such that $\ptl D_R$ is smooth, we obtain
\begin{equation*}
\Vol_h(D_R)\geq (\leq)\,\,\frac{\Vol_f(B^w_R)}{\Vol_f(\ptl B^w_R)}\,\int_{\ptl B_R}|\nabla_P r|\,da_h,
\end{equation*}
where $B^w_R$ is the metric ball of radius $R$ centered at the pole $o_w$ in $(M^n_w,g_w)$.
\end{theorem}

\begin{proof}
Fix $R>0$ such that $\ptl D_R$ is smooth. The function $\phi_R:[0,R]\to\rr$ given by
\begin{equation*}
\phi_{R}(s):=\int_s^R q_{w,f}(t)\,dt
\end{equation*}
is $C^2$ and satisfies the differential equation
\begin{equation}
\label{poisson5}
\phi_R''(s)+\phi_R'(s)\left((n-1)\,\frac{w'(s)}{w(s)}+\psi(s)+\var(s)\right)=-1 \ \text{ in } [0,R].
\end{equation}

We define $v:=\phi_{R}\circ r$, which is a radial function in $C^2(\overline{D}_R)$. Since $\phi'_R\leq 0$, by using Theorem~\ref{Laplacian-sub-infty} and the estimates in (ii), we infer that
\begin{displaymath}
\Delta^h_P\,v\geq (\leq)\left(\phi_{R}''(r)-\phi_{R}'(r)\,\frac{w'(r)}{w(r)}\right)|\nabla_P r|^2 + \phi_{R}'(r)\,\left(\frac{n\,w'(r)}{w(r)} +\varphi(r)+\psi(r)\right).
\end{displaymath}
Observe that \eqref{poisson5} and the balance condition in (iii) imply that
\[
\phi_{R}''(r)-\phi_{R}'(r)\,\frac{w'(r)}{w(r)}=-1-\phi_{R}'(r)\,\bigg(n\,\frac{w'(r)}{w(r)}+\psi(r)+\varphi(r)\bigg)\leq (\geq)\,\,0
\]
in $D_R-\{o\}$. As $|\nabla_P r|\leq 1$, we conclude that
\begin{displaymath}
\Delta^h_P\,v\geq (\leq)\,\phi_{R}''(r)+\phi_{R}'(r)\,\left((n-1)\,\frac{w'(r)}{w(r)}+\varphi(r)+\psi(r)\right)=-1
\end{displaymath}
in $D_R$. Finally, we integrate and apply the divergence theorem to get
\begin{align*}
\Vol_h(D_R)&\geq (\leq)\, -\int_{D_R}\divv^h\nabla_P v\,dv_h=-\int_{\partial D_R} \bigg\langle\nabla_P v,\frac{\nabla_P r}{|\nabla_P r|}\bigg\rangle\,d
a_h
\\
&=q_{w,f}(R)\,\int_{\ptl D_R}|\nabla_P r|\,da_h=\frac{\Vol_f(B^w_R)}{\Vol_f(\ptl B^w_R)}\,\int_{\ptl B_R}|\nabla_P r|\,da_h,
\end{align*}
where we have used that $\phi_R'(s)=-q_{w,f}(s)$ and that the outer conormal vector along $\ptl D_R$ is $\frac{\nabla_P r}{|\nabla_P r|}$. This proves the claim.
\end{proof}

\begin{remark}
In the case where $\text{Sec}(\sigma_p)\geq-\frac{w''(r)}{w(r)}$ the theorem extends to weighted manifolds a comparison of Markvorsen and the second author~\cite{MP5}. In the case $\text{Sec}(\sigma_p)\leq-\frac{w''(r)}{w(r)}$ the fact that $|\nabla_Pr|\leq 1$ on $P-\{o\}$ leads to the estimate
\[
\frac{\Vol_h(D_R)}{\Vol_h(\ptl D_R)}\leq\frac{\Vol_f(B^w_R)}{\Vol_f(\ptl B^w_R)}.
\]
This generalizes to a weighted context a result of the second author~\cite{Pa} for minimal submanifolds of Cartan-Hadamard manifolds.
\end{remark}

Next, we provide some criteria for the $h$-parabolicity or $h$-hyperbolicity of non-compact submanifolds properly immersed in a weighted manifold with bounded radial sectional curvatures.

Our first result is an extension to the weighted setting of previous theorems for Riemannian manifolds by Esteve and the second author~\cite{esteve-palmer}, and by Markvorsen and the second author~\cite{MP3}. We note that the particular situation of rotationally symmetric manifolds with weights was analyzed by the authors in Theorems 3.2 and 3.3 of \cite{HPR1}.

\begin{theorem}
\label{parabolicity-sub-h''}
Let $(M^m,g,e^h)$ be a weighted manifold, $P^n$ a non-compact submanifold properly immersed in $M$, $r:M\rightarrow [0,\infty[$ the distance function from a pole $o\in M$, and $w(s)$ a smooth function such that $w(0)=0$, $w'(0)=1$
and $w(s)>0$ for all $s>0$. Suppose that the following conditions are fulfilled:
\begin{itemize}
\item[(i)] For any $p\in M-\{o\}$ and any plane $\sigma_p\subseteq T_pM$ containing $(\nabla r)_p$, we have
\begin{displaymath}
{\rm Sec}(\sigma_p) \geq (\leq) -\frac{w''(r)}{w(r)}.
\end{displaymath}
\item[(ii)] There exist $\rho>0$ and continuous functions $\psi(s),\varphi(s)$, such that $\ptl D_\rho$ is smooth and
\begin{displaymath}
\langle \nabla h,\nabla r\rangle\leq (\geq)\,\psi(r), \quad \langle \ovh^h,\nabla r\rangle \leq (\geq)\,  \varphi(r) \ \text{ on }P-D_\rho.
\end{displaymath}
\item[(iii)] In $P-D_\rho$ the bounding functions verify
\begin{displaymath}
\psi(r)+\varphi(r)\leq (\geq)\, -n\,\frac{w'(r)}{w(r)}\quad {\textrm {$($balance condition$)$}}.
\end{displaymath}
\end{itemize}
Then, we obtain
\begin{equation*}
\C^h_P(D_\rho)\leq (\geq)\,\,\frac{\C^f(B^w_\rho)}{\Vol_f(\ptl B^w_\rho)}\,\int_{\ptl D_R}|\nabla_Pr|\,da_h,
\end{equation*}
where $\C^f(B^w_\rho)$ denotes the weighted capacity of the metric ball $B^w_\rho$ in a weighted $(w,f)$-model space $(M^n_w,g_w,e^{f(r)})$ with
$f(r):=\int_\rho^r (\psi(s)+\varphi(s))\,ds$ for any $r\geq\rho$.

Moreover, if
\begin{equation}
\label{cond-para-sub-h''}
\int_\rho^\infty w^{1-n}(s)\, e^{-f(s)}\,ds=(<)\,\infty,
\end{equation}
then $P$ is $h$-parabolic $($$h$-hyperbolic$)$.
\end{theorem}

\begin{proof}
By using Sard's Theorem we can suppose that $\nabla_P r\neq 0$ along $\ptl D_\rho$. Take any number $R>\rho$ such that $\ptl D_R$ is smooth. Let us consider the extrinsic annulus $A^P_{\rho,R}:=D_R-\overline{D}_\rho$ and the function $\phi_{\rho,R,f}:[\rho,R]\to\rr$ defined in \eqref{fsolmodel}, i.e., the $f$-capacity potential of $(\overline{B}^w_\rho,B^w_R)$ in the $n$-dimensional weighted $(w,f)$-model space $(M^n_w,g_w,e^{f(r)})$. This function is the solution to the problem \eqref{eqfDirModel2}; in particular, it satisfies \eqref{eq:sabika} by replacing $m$ with $n$. The composition $v:=\phi_{\rho,R,f} \circ r$ defines a smooth function in $A^P_{\rho,R}$.

Since $\phi_{\rho,R,f}'(r) \leq 0$, by using Theorem~\ref{Laplacian-sub-infty} and the boundedness assumptions (ii) in the statement, we get this comparison in $A^P_{\rho,R}$
\begin{displaymath}
\Delta^h_P\,v\geq (\leq)\,\left(\phi_{\rho,R,f}''(r)-\phi_{\rho,R,f}'(r)\,\frac{w'(r)}{w(r)}\right)|\nabla_P r|^2 + \phi_{\rho,R,f}'(r)\,\left(\frac{n\,w'(r)}{w(r)} +\varphi(r)+\psi(r)\right).
\end{displaymath}
On the other hand, by taking into account \eqref{eq:sabika} and the balance condition in (iii), it follows that
\[
\phi_{\rho,R,f}''(r)-\phi_{\rho,R,f}'(r)\,\frac{w'(r)}{w(r)}=-\phi_{\rho,R,f}'(r)\,\bigg(n\,\frac{w'(r)}{w(r)}+\psi(r)+\varphi(r)\bigg)\leq (\geq)\,\, 0
\]
in $A^P_{\rho,R}$. As $|\nabla_P r|\leq 1$, we conclude that
\begin{displaymath}
\Delta^h_P\,v\geq (\leq)\,\, \phi_{\rho,R,f}''(r)+\phi_{\rho,R,f}'(r)\,\left((n-1)\,\frac{w'(r)}{w(r)}+\varphi(r)+\psi(r)\right)=0
=\Delta^h_P\,u,
\end{displaymath}
where $u$ is the $h$-capacity potential of the capacitor $(\overline{D}_\rho,D_R)$ in $P$. Since $u=v$ on $\partial A^P_{\rho,R}$, by applying the maximum principle and the Hopf boundary point lemma in Theorem~\ref{th:mp}, we deduce that $\frac{\ptl u}{\ptl\nu}< (>)\,\frac{\ptl v}{\ptl\nu}$ on $\ptl D_\rho$, where $\nu$ is the outer unit normal along $\ptl A^P_{\rho,R}$, which coincides with the unit normal $\frac{\nabla_Pu}{|\nabla_Pu|}=\frac{\nabla_P v}{|\nabla_Pv|}$ along $\ptl D_\rho$ pointing into $D_\rho$. From \eqref{eq:capint}, we obtain
\begin{align*}
\C^h_P(D_\rho,D_R)&=\int_{\partial D_\rho} |\nabla_Pu|\,da_h\leq (\geq) \int_{\partial D_\rho} |\nabla_P v|\,da_h
\\
&=|\phi'_{\rho,R,f}(\rho)|\int_{\partial D_\rho} |\nabla_P r|\,da_h 
\\
&=\frac{\C^f(B^w_\rho,B^w_R)}{\Vol_f(\ptl B^w_\rho)}\int_{\partial D_\rho} |\nabla_P r|\,da_h .
\end{align*}
Hence, the desired comparison comes from above by taking limits when $R\to\infty$.
Moreover, if \eqref{cond-para-sub-h''} holds, then $\C^f(B^w_\rho)=\C_P^h(D_\rho)=0$ (resp. $\C^f(B^w_\rho)>0$ and $\C^h_P(D_R)>0$), so that $P$ is $h$-parabolic (resp. $h$-hyperbolic) by Theorem~\ref{theorGrig}.
\end{proof}

\begin{remark}
Under the hypotheses corresponding to the case $\text{Sec}(\sigma_p)\geq-\frac{w''(r)}{w(r)}$ we can deduce that
\[
\frac{\C^h_P(D_\rho)}{{\rm Vol}_h(\partial D_\rho)} \leq \frac{\C^f(B^w_\rho)}{\Vol_f(\ptl B^w_\rho)}.
\]
\end{remark}

Next, we will deduce some consequences of the previous result for submanifolds with bounded weighted mean curvature.

\begin{corollary}
\label{corext1}
Let $(M^m,g, e^h)$ be a weighted manifold, $P^n$ a non-compact submanifold properly immersed in $M$, $r:M\rightarrow [0,\infty[$ the distance function from a pole $o\in M$, and $w(s)$ a smooth function such that $w(0)=0$, $w'(0)=1$ and $w(s)>0$ for all $s>0$. Suppose that the following conditions are fulfilled:
\begin{itemize}
\item[(i)] The function $w$ satisfies $\int_0^\infty w(s)\,ds=\infty$ $($resp. $\int_0^\infty w(s)\,ds<\infty$$)$ and $\frac{w'(s)}{w(s)}$ is bounded at infinity.
\item[(ii)] For any $p\in M-\{o\}$ and any plane $\sigma_p\subseteq T_pM$ containing $(\nabla r)_p$, we have
\begin{displaymath}
{\rm Sec}(\sigma_p) \geq (\leq) -\frac{w''(r)}{w(r)}.
\end{displaymath}
\item[(iii)] There exist $\rho>0$ and a continuous function $\psi(s)$ with $\psi(s) \rightarrow -\infty$ $($resp. $\psi(s)\rightarrow\infty$$)$ when $s\to\infty$, such that $\ptl D_\rho$ is smooth and
\begin{displaymath}
\langle \nabla h,\nabla r\rangle\leq (\geq)\,\psi(r) \ \text{ on }P-D_\rho. 
\end{displaymath}
\end{itemize}

In these conditions, if $P$ has bounded $h$-mean curvature, then $P$ is $h$-parabolic $($resp. $h$-hyperbolic$)$.
\end{corollary}

\begin{proof} 
We check the hypotheses in Theorem~\ref{parabolicity-sub-h''}.
Choose $c>0$ such that $|\ovh^h|\leq c$ on $P$. The Cauchy-Schwarz inequality implies that $\escpr{\ovh^h,\nabla r}\leq c$ (resp. $\escpr{\ovh^h,\nabla r}\geq -c$) on $P-\{o\}$. On the other hand, since $\frac{w'(s)}{w(s)}$ is bounded at infinity and $\psi(s) \to-\infty$ (resp. $\psi(s)\to\infty$) when $s\to\infty$, by changing $\rho$ if necessary, we can suppose that
\begin{equation*}
n\,\frac{w'(r)}{w(r)}+\psi(r)+c\leq 0 \ \  \text{\bigg(resp. } n\,\frac{w'(r)}{w(r)}+\psi(r)-c\geq 0 \text{\bigg)} \ \text{ on } P-D_\rho,
\end{equation*}
so that the balance condition is satisfied. Consider the function $f(s):=\int_\rho^s (\psi(t)+c)\,dt$ (resp. $f(s):=\int_\rho^s (\psi(t)-c)\,dt$).
By integrating the inequality above, we obtain
\[
f(s)\leq (\geq)-n\,\int_\rho^s\frac{w'(t)}{w(t)}\,dt=-n\,\ln\bigg(\frac{w(s)}{w(\rho)}\bigg),
\]
and so
\[
e^{-f(s)}\geq (\leq)\,\,\frac{w^n(s)}{w^n(\rho)}, \ \text{ for any }s\geq\rho.
\]
From here we have
\[
\int_\rho^\infty w^{1-n}(s)\,e^{-f(s)}\,ds\geq (\leq)\,\,\frac{1}{w^n(\rho)}\,\int_\rho^\infty w(s)\,ds,
\]
so that the condition in \eqref{cond-para-sub-h''} holds since $\int_0^\infty w(s)\,ds=\infty$ (resp. $\int_0^\infty w(s)\,ds<\infty$).
We conclude that $P$ is $h$-parabolic (resp. $h$-hyperbolic).
\end{proof}

Also as a consequence of Theorem \ref{parabolicity-sub-h''} we can extend to a weighted setting a result of S. Markvorsen and the second author \cite{MP2} ensuring that, in a Cartan-Hadamard manifold with sectional curvatures bounded from above by $b\leq 0$, the $n$-dimensional complete minimal and properly immersed submanifolds are hyperbolic if either $b<0$ and $n\geq 2$, or $b=0$ and $n \geq 3$. This statement is the particular case $h=0$ of the next corollary.

\begin{corollary} 
\label{hiperbolicity-MP} 
Let $(M^m,g)$ be a Cartan-Hadamard manifold, i.e., a complete and simply connected Riemannian manifold such that
\[
{\rm Sec} (\sigma_p) \leq b \leq 0,
\] 
for any plane $\sigma_p\subseteq T_pM$ and any point $p\in M$. Denote by $r:M\to[0,\infty[$ the distance function from a fixed point $o\in M$. Given a weight $e^h$ in $M$, suppose that there exist $n\in\mathbb{N}$ with $n\geq 2$, and constants $\rho,\epsilon>0$, such that
\[
\begin{cases}
\escpr{\nabla h,\nabla r}\geq-\frac{n-2-\epsilon}{r} \ \text{ in } M-B_\rho \ \text{ if } b=0, 
\\
\escpr{\nabla h,\nabla r}\geq-(n-1-\epsilon)\,\sqrt{-b}\,\coth(\sqrt{-b}\,r) \ \text{ in } M-B_\rho \ \text{ if } b<0.
\end{cases}
\]
Then, any non-compact $h$-minimal submanifold $P^n$ properly immersed in $M$ is $h$-hyperbolic.
\end{corollary}

\begin{proof} 
For a submanifold $P$ in the conditions of the statement we check that the hypotheses in Theorem~\ref{parabolicity-sub-h''} are satisfied. 

In case $b=0$ we consider the functions $w,\psi,\varphi:[0,\infty[\to\rr$ defined by $w(s):=s$, $\psi(s):=-\frac{n-2-\epsilon}{s}$ and $\varphi(s):=0$. Observe that
\[
\psi(r)+\varphi(r)+n\,\frac{w'(r)}{w(r)}=\frac{\epsilon+2}{r}>0,
\]
so that the balance condition holds. On the other hand, a straightforward computation shows that
\[
f(s):=\int_\rho^s\psi(t)\,dt=-(n-2-\epsilon)\,\ln\bigg(\frac{s}{\rho}\bigg),
\]
and so
\[
\int_\rho^\infty w^{1-n}(s)\,e^{-f(s)}\,ds=\frac{1}{\rho^{n-2-\eps}}\,\int_{\rho}^\infty s^{-(\eps+1)}\,ds<\infty,
\]
which is the condition in \eqref{cond-para-sub-h''}. From here we conclude that $P$ is $h$-hyperbolic. 

In case $b<0$ we reason in a similar way with the functions $w,\psi,\varphi:[0,\infty[\to\rr$ given by $w(s):=\frac{1}{\sqrt{-b}}\sinh(\sqrt{-b}\, s)$, $\psi(s):=-(n-1-\epsilon)\,\sqrt{-b}\,\coth(\sqrt{-b}\,s)$ and $\varphi(s):=0$.
\end{proof}

\begin{remark}
In Theorem~\ref{Laplacian-sub-infty} and the results of this subsection we assume bounds on the Riemannian sectional curvatures. As we noted in Remark~\ref{re:papa}, these hypotheses hold when we assume the same bounds on $\text{Sec}^h_\infty$ under an additional condition for the sign of $(\text{Hess}\,h)(\nabla r,\nabla r)$.
\end{remark}

\begin{remark}
As we pointed out in Remark~\ref{re:salinas}, the condition in \eqref{cond-para-sub-h''} is equivalent by Remark~\ref{Alhfors} to the weighted parabolicity (resp. hyperbolicity) of the corresponding $n$-dimensional weighted comparison model. Hence, Theorem~\ref{parabolicity-sub-h''} shows that the submanifold $P$ is $h$-parabolic (resp. $h$-hyperbolic) provided the weighted $(w,f)$-model space is $f$-parabolic (resp. $f$-hyperbolic).
\end{remark}

\begin{remark}[Comparisons under a lower bound on $\text{Sec}^h_\infty$]
By following the proofs of Theorems~\ref{th:simpson} and \ref{parabolicity-sub-h''} it is possible to derive a volume comparison and a parabolicity criterion for a weighted manifold $(M^m,g,e^h)$, where $h$ is a radial non-decreasing weight satisfying that
\[
\text{Sec}^h_\infty(\sigma_p)\geq-\frac{w''(r)}{w(r)}.
\]
The starting point for these comparisons is the inequality \eqref{hessianineqinfty} in Theorem~\ref{th:hessianineqr}, from which the same estimate for $\Delta^h_P(F\circ r)$ as in Theorem~\ref{Laplacian-sub-infty} can be deduced. The details are left to the reader.
\end{remark}

\subsection{Weighted parabolicity under a lower bound on the $q$-weighted sectional curvatures.}
\label{extsect2}\

We finally show a parabolicity criterion by assuming a lower bound on some $q$-weighted sectional curvature. The key ingredients to prove it are the comparison in Theorem~\ref{laplace-ineq-q} for the weighted Laplacian and the use, as in Theorem~\ref{comp:capacity}, of a second order operator over radial functions that coincides with the weighted Laplacian in some weighted model space when $q\in\mathbb{N}$.

\begin{theorem} 
\label{parabolicity-sub-gen-q}
Let $(M^m,g, e^h)$ be a weighted manifold, $P^n$ a non-compact submanifold properly immersed in $M$, $r:M\to[0,\infty[$ the distance function from a pole $o\in M$, and $w(s)$ a smooth function such that $w(0)=0$, $w'(0)=1$ and $w(s)>0$ for all $s>0$. Suppose that the following conditions are fulfilled:
\begin{itemize}
\item[(i)] There is $q>0$ such that, for any $p\in M-\{o\}$ and any plane $\sigma_p\subseteq T_pM$ containing $(\nabla r)_p$, we have
\begin{displaymath}
{\rm Sec}^{h}_q(\sigma_p)\geq-\,\frac{m+q-1}{m-1}\,\frac{w''(r)}{w(r)}.\end{displaymath}
\item[(ii)] There exist $\rho>0$ and continuous functions $\psi(s),\varphi(s)$, such that $\ptl D_\rho$ is smooth and
\begin{displaymath}
\langle \nabla h,\nabla r\rangle\leq\psi(r), \quad \langle \ovh^h,\nabla r\rangle \leq  \varphi(r) \ \text{ on }P-D_\rho.
\end{displaymath}
\item[(iii)] In $P-D_\rho$ the bounding functions verify
\begin{displaymath}
\frac{m-n-1}{m-1}\,\psi(r)+\varphi(r)\leq -n\,\,\frac{m+q-1}{m-1}\,\frac{w'(r)}{w(r)}\quad {\textrm {$($balance condition$)$}}.
\end{displaymath}
\end{itemize}
Then, we obtain 
\begin{equation*}
\frac{\C^h_P(D_\rho)}{{\rm Vol}_h(\partial D_\rho)} \leq
\frac{w^{(1-n)\,\frac{m+q-1}{m-1}}(\rho)\,
e^{-f(\rho)}}{\int_\rho^{\infty}w^{(1-n)\,\frac{m+q-1}{m-1}}(s)\,
e^{-f(s)}\,ds},
\end{equation*}
where
$f(r):=\int_\rho^r\left(\frac{m-n}{m-1}\,\psi(s)+\varphi(s)\right)ds$ for any $r\geq\rho$.

Moreover, if
\begin{equation}
\label{cond-para-sub-q}
\int_\rho^{\infty}w^{(1-n)\,\frac{m+q-1}{m-1}}(s)\,e^{-f(s)}\,ds=\infty,
\end{equation}
then $P$ is $h$-parabolic.
\end{theorem}

\begin{proof} 
We define a second order differential operator $L$ acting on smooth functions $F:[\rho,\infty[\to\rr$ by
\begin{equation*}
L(F(s)):=
F''(s)+F'(s)\left((n-1)\,\frac{m+q-1}{m-1}\,\frac{w'(s)}{w(s)}+\frac{m-n}{m-1}\,\psi(s)+\varphi(s)\right).
\end{equation*}
For any $R>\rho$, it is easy to see that the unique solution of equation $L(F(s))=0$ in $[\rho,R]$ with boundary conditions $F(\rho)=1$ and $F(R)=0$
is given by the function
\begin{equation*}
\phi^{L}_{\rho,R}(s):=\bigg(\int_s^R w^{(1-n)\,\frac{m+q-1}{m-1}}(t)\,
e^{-f(t)} dt\bigg)\,\bigg(\int_\rho^{R}w^{(1-n)\,\frac{m+q-1}{m-1}}(t)\,e^{-f(t)}
dt\bigg)^{-1}.
\end{equation*}

Now, we consider the radial function $v:=\phi^{L}_{\rho,R} \circ r$ defined in the extrinsic annulus $A^P_{\rho,R}:=D_R-\overline{D}_\rho$ of $P$.
Since $(\phi^{L}_{\rho,R})'(r) \leq 0$, by using Theorem
\ref{laplace-ineq-q} together with the boundedness assumptions in (ii), we get this inequality in $A^P_{\rho,R}$
\begin{equation*}
\begin{aligned}
\Delta_P^h\,v\geq &\left((\phi^{L}_{\rho,R})''(r)-\frac{(\phi^{L}_{\rho,R})'(r)}{m-1}\left((m+q-1)\,\frac{w'(r)}{w(r)}-\psi(r)\right)\right)|\nabla_Pr|^2
\\
&+(\phi^{L}_{\rho,R})'(r)\left(n\,\frac{m+q-1}{m-1}\,\frac{w'(r)}{w(r)}+\frac{m-n-1}{m-1}\,\psi(r)+\varphi(r)\right).
\end{aligned}
\end{equation*}
On the other hand, from equality $L(\phi^L_{\rho,R,f}(r))=0$ and the balance condition in (iii), it follows that 
\begin{align*}
(\phi^{L}_{\rho,R})''(r)&-\frac{(\phi^{L}_{\rho,R})'(r)}{m-1}
\left((m+q-1)\,\frac{w'(r)}{w(r)}-\psi(r)\right)
\\
&=(\phi^{L}_{\rho,R})'(r)\left(\frac{1-m-n}{m-1}\,\psi(r)-n\,\frac{m+q-1}{m-1}\,\frac{w'(r)}{w(r)}-\varphi(r)\right)\leq 0.
\end{align*}
Thus, since $|\nabla_P r|\leq 1$, we conclude that
\begin{displaymath}
\Delta^h_P\,v\geq L(\phi^{L}_{\rho,R}(r))=0=\Delta^h_P\,u,
\end{displaymath}
where $u$ is the $h$-capacity potential of the capacitor
$(\overline{D}_\rho,D_R)$ in $P$. From this point the claim follows with the same arguments as in the proof of Theorem \ref{parabolicity-sub-h''}. 
\end{proof}

As a consequence of Theorem~\ref{parabolicity-sub-gen-q} we deduce the following result for weighted minimal hypersurfaces.

\begin{corollary} 
Let $(M^m,g, e^h)$ be a weighted manifold, $r:M\to[0,\infty[$ the distance function from a pole $o\in M$, and $w(s)$ a smooth function such that $w(0)=0$, $w'(0)=1$ and $w(s)>0$ for all $s>0$. Suppose that the following conditions are fulfilled:
\begin{itemize}
\item[(i)] There exists $\rho>0$ such that $w:[\rho,\infty[\to\rr$ is a non-increasing function.
\item[(ii)] There is $q>0$ such that, for any $p\in M-\{o\}$ and any plane $\sigma_p\subseteq T_pM$ containing $(\nabla r)_p$, we have
\begin{displaymath}
{\rm Sec}^{h}_q(\sigma_p)\geq-\,\frac{m+q-1}{m-1}\,\frac{w''(r)}{w(r)}.
\end{displaymath}
\item[(iii)] $h$ is a non-positive radial function.
\end{itemize}
Then, any non-compact $h$-minimal hypersurface $P$ properly immersed in $M$ is $h$-parabolic.
\end{corollary}

\begin{proof}
We apply Theorem \ref{parabolicity-sub-gen-q}. Since $h$ is radial and $P$ is $h$-minimal, we consider the functions $\psi(r):=h'(r)$ and
$\varphi(r):=0$. As $n=m-1$, the balance condition reads
\[
0\leq -(m+q-1)\,\frac{w'(r)}{w(r)},
\]
which holds on $P-D_\rho$ because $w'\leq 0$ in $[\rho, \infty)$. On the other hand
\begin{equation*}
\begin{aligned}
&\int_\rho^{\infty}w^{(1-n)\,\frac{m+q-1}{m-1}}(s)\,e^{-f(s)}\,ds
=e^{\frac{h(\rho)}{m-1}}\int_\rho^{\infty} w^{(2-m)\,\frac{m+q-1}{m-1}}(s)\,
e^{-\frac{h(s)}{m-1}}\,ds
\\ 
&\geq e^{\frac{h(\rho)}{m-1}}\,w^{(2-m)\,\frac{m+q-1}{m-1}}(\rho)\,\int_\rho^{\infty}e^{-\frac{h(s)}{m-1}}\,ds=\infty,
\end{aligned}
\end{equation*}
because $w$ is nonincreasing in $[\rho,\infty)$ and $h\leq 0$. So, condition \eqref{cond-para-sub-q} holds and $P$ is $h$-parabolic.
\end{proof}

\begin{remark}
Under the assumptions of Theorem~\ref{parabolicity-sub-gen-q} we can obtain a weighted volume comparison for extrinsic balls of submanifolds in the line of those in Theorems~\ref{th:simpson} and \ref{comp:qcapacity}. The details are left to the reader.
\end{remark}

\bibliographystyle{plain}

\begin{thebibliography}{10}

\bibitem{Ahl}
L.V. Ahlfors.
\newblock Sur le type d'une surface de {R}iemann.
\newblock {\em C. R. Acad. Sci. Paris}, 201:30--32, 1935.

\bibitem{be}
D.~Bakry and M.~{\'E}mery.
\newblock Diffusions hypercontractives.
\newblock In {\em S\'eminaire de probabilit\'es, {XIX}, 1983/84}, volume 1123
  of {\em Lecture Notes in Math.}, pages 177--206. Springer, Berlin, 1985.

\bibitem{MR2243959}
D.~Bakry and Z.~Qian.
\newblock Volume comparison theorems without {J}acobi fields.
\newblock In {\em Current trends in potential theory}, volume~4 of {\em Theta
  Ser. Adv. Math.}, pages 115--122. Theta, Bucharest, 2005.

\bibitem{bayle-thesis}
V.~Bayle.
\newblock {\em Propri\'et\'es de concavit\'e du profil isop\'erim\'etrique et
  applications}.
\newblock PhD thesis, Institut Fourier (Grenoble), 2003.

\bibitem{homostable}
A.~Ca{\~n}ete and C.~Rosales.
\newblock Compact stable hypersurfaces with free boundary in convex solid cones
  with homogeneous densities.
\newblock {\em Calc. Var. Partial Differential Equations}, 51(3-4):887--913,
  2014.

\bibitem{chavel}
I.~Chavel.
\newblock {\em Riemannian geometry}, volume~98 of {\em Cambridge Studies in
  Advanced Mathematics}.
\newblock Cambridge University Press, Cambridge, second edition, 2006.
\newblock A modern introduction.

\bibitem{dcriem}
M.~P. do~Carmo.
\newblock {\em Riemannian geometry}.
\newblock Mathematics: Theory \& Applications. Birkh\"auser Boston Inc.,
  Boston, MA, 1992.
\newblock Translated from the second Portuguese edition by Francis Flaherty.

\bibitem{esteve-palmer}
A.~Esteve and V.~Palmer.
\newblock On the characterization of parabolicity and hyperbolicity of
  submanifolds.
\newblock {\em J. Lond. Math. Soc. (2)}, 84(1):120--136, 2011.

\bibitem{gilbarg-trudinger}
D.~Gilbarg and N.~S. Trudinger.
\newblock {\em Elliptic partial differential equations of second order}.
\newblock Springer-Verlag, Berlin-New York, 1977.
\newblock Grundlehren der Mathematischen Wissenschaften, Vol. 224.

\bibitem{GreW}
R.~E. Greene and H.~Wu.
\newblock {\em Function theory on manifolds which possess a pole}, volume 699
  of {\em Lecture Notes in Mathematics}.
\newblock Springer, Berlin, 1979.

\bibitem{grigoryan}
A.~Grigor'yan.
\newblock Analytic and geometric background of recurrence and non-explosion of
  the {B}rownian motion on {R}iemannian manifolds.
\newblock {\em Bull. Amer. Math. Soc. (N.S.)}, 36(2):135--249, 1999.

\bibitem{grigoryan-escape}
A.~Grigor'yan.
\newblock Escape rate of {B}rownian motion on {R}iemannian manifolds.
\newblock {\em Appl. Anal.}, 71(1-4):63--89, 1999.

\bibitem{grigoryan2}
A.~Grigor'yan.
\newblock Heat kernels on weighted manifolds and applications.
\newblock In {\em The ubiquitous heat kernel}, volume 398 of {\em Contemp.
  Math.}, pages 93--191. Amer. Math. Soc., Providence, RI, 2006.

\bibitem{Gri2}
A.~Grigor{\cprime}yan.
\newblock Heat kernels on weighted manifolds and applications.
\newblock In {\em The ubiquitous heat kernel}, volume 398 of {\em Contemp.
  Math.}, pages 93--191. Amer. Math. Soc., Providence, RI, 2006.

\bibitem{gri-book}
A.~Grigor'yan.
\newblock {\em Heat kernel and analysis on manifolds}, volume~47 of {\em AMS/IP
  Studies in Advanced Mathematics}.
\newblock American Mathematical Society, Providence, RI; International Press,
  Boston, MA, 2009.

\bibitem{gri-masa}
A.~Grigor'yan and J.~Masamune.
\newblock Parabolicity and stochastic completeness of manifolds in terms of the
  {G}reen formula.
\newblock {\em J. Math. Pures Appl. (9)}, 100(5):607--632, 2013.

\bibitem{Gri3}
A.~Grigor'yan and L.~Saloff-Coste.
\newblock Hitting probabilities for {B}rownian motion on {R}iemannian
  manifolds.
\newblock {\em J. Math. Pures Appl. (9)}, 81(2):115--142, 2002.

\bibitem{gromov-GAFA}
M.~Gromov.
\newblock Isoperimetry of waists and concentration of maps.
\newblock {\em Geom. Funct. Anal.}, 13(1):178--215, 2003.

\bibitem{HPR1}
A.~Hurtado, V.~Palmer, and C.~Rosales.
\newblock Parabolicity criteria and characterization results for submanifolds
  of bounded mean curvature in model manifolds with weights.
\newblock arXiv:1805.10055, May 2018.

\bibitem{I1}
K.~Ichihara.
\newblock Curvature, geodesics and the {B}rownian motion on a {R}iemannian
  manifold {I}. {R}ecurrence properties.
\newblock {\em Nagoya Math. J.}, 87:101--114, 1982.

\bibitem{JK}
L.~Jorge and D.~Koutroufiotis.
\newblock An estimate for the curvature of bounded submanifolds.
\newblock {\em Amer. J. Math.}, 103(4):711--725, 1981.

\bibitem{MR3641482}
L.~Kennard and W.~Wylie.
\newblock Positive weighted sectional curvature.
\newblock {\em Indiana Univ. Math. J.}, 66(2):419--462, 2017.

\bibitem{MR3897040}
L.~Kennard, W.~Wylie, and D.~Yeroshkin.
\newblock The weighted connection and sectional curvature for manifolds with
  density.
\newblock {\em J. Geom. Anal.}, 29(1):957--1001, 2019.

\bibitem{lich1}
A.~Lichnerowicz.
\newblock Vari{\'e}t{\'e}s riemanniennes {\`a} tenseur {C} non n{\'e}gatif.
\newblock {\em C. R. Acad. Sci. Paris S{\'e}r. A-B}, 271:A650--A653, 1970.

\bibitem{lich2}
A.~Lichnerowicz.
\newblock Vari{\'e}t{\'e}s k{\"a}hl{\'e}riennes {\`a} premi{\`e}re classe de
  {C}hern non negative et vari{\'e}t{\'e}s riemanniennes {\`a} courbure de
  {R}icci g{\'e}n{\'e}ralis{\'e}e non negative.
\newblock {\em J. Differential Geom.}, 6:47--94, 1971/72.

\bibitem{lott}
J.~Lott.
\newblock Some geometric properties of the {B}akry-\'{E}mery-{R}icci tensor.
\newblock {\em Comment. Math. Helv.}, 78(4):865--883, 2003.

\bibitem{MRS}
L.~Mari, M.~Rigoli, and A.~G. Setti.
\newblock Keller-{O}sserman conditions for diffusion-type operators on
  {R}iemannian manifolds.
\newblock {\em J. Funct. Anal.}, 258(2):665--712, 2010.

\bibitem{MaP1}
S.~Markvorsen and V.~Palmer.
\newblock Generalized isoperimetric inequalities for extrinsic balls in minimal
  submanifolds.
\newblock {\em J. Reine Angew. Math.}, 551:101--121, 2002.

\bibitem{MP2}
S.~Markvorsen and V.~Palmer.
\newblock Transience and capacity of minimal submanifolds.
\newblock {\em Geom. Funct. Anal.}, 13(4):915--933, 2003.

\bibitem{MP3}
S.~Markvorsen and V.~Palmer.
\newblock How to obtain transience from bounded radial mean curvature.
\newblock {\em Trans. Amer. Math. Soc.}, 357(9):3459--3479, 2005.

\bibitem{MP4}
S.~Markvorsen and V.~Palmer.
\newblock Torsional rigidity of minimal submanifolds.
\newblock {\em Proc. London Math. Soc. (3)}, 93(1):253--272, 2006.

\bibitem{MP5}
S.~Markvorsen and V.~Palmer.
\newblock Extrinsic isoperimetric analysis of submanifolds with curvatures
  bounded from below.
\newblock {\em J. Geom. Anal.}, 20(2):388--421, 2010.

\bibitem{morganmyers}
F.~Morgan.
\newblock Myers' theorem with density.
\newblock {\em Kodai Math. J.}, 29(3):455--461, 2006.

\bibitem{gmt}
F.~Morgan.
\newblock {\em Geometric measure theory. {A} beginner's guide}.
\newblock Elsevier/Academic Press, Amsterdam, fourth edition, 2009.

\bibitem{MR3197658}
O.~Munteanu and J.~Wang.
\newblock Geometry of manifolds with densities.
\newblock {\em Adv. Math.}, 259:269--305, 2014.

\bibitem{oneill}
B.~O'Neill.
\newblock {\em Semi-{R}iemannian geometry}, volume 103 of {\em Pure and Applied
  Mathematics}.
\newblock Academic Press, Inc. [Harcourt Brace Jovanovich, Publishers], New
  York, 1983.
\newblock With applications to relativity.

\bibitem{Pa}
V.~Palmer.
\newblock Isoperimetric inequalities for extrinsic balls in minimal
  submanifolds and their applications.
\newblock {\em J. London Math. Soc. (2)}, 60(2):607--616, 1999.

\bibitem{Pa2}
V.~Palmer.
\newblock {\em On deciding whether a submanifold is parabolic of hyperbolic
  using its mean curvature}.
\newblock Simon Stevin Institute for Geometry, Tilburg, The Netherlands, 2010.
\newblock Simon Stevin Transactions on Geometry, vol 1.

\bibitem{Pe}
P.~Petersen.
\newblock {\em Riemannian geometry}, volume 171 of {\em Graduate Texts in
  Mathematics}.
\newblock Springer, New York, second edition, 2006.

\bibitem{PRRS}
S.~Pigola, M.~Rigoli, M.~Rimoldi, and A.~G. Setti.
\newblock Ricci almost solitons.
\newblock {\em Ann. Sc. Norm. Super. Pisa Cl. Sci. (5)}, 10(4):757--799, 2011.

\bibitem{prs-book}
S.~Pigola, M.~Rigoli, and A.~G. Setti.
\newblock {\em Vanishing and finiteness results in geometric analysis}, volume
  266 of {\em Progress in Mathematics}.
\newblock Birkh\"{a}user Verlag, Basel, 2008.
\newblock A generalization of the Bochner technique.

\bibitem{qian2}
Z.~Qian.
\newblock On conservation of probability and the {F}eller property.
\newblock {\em Ann. Probab.}, 24(1):280--292, 1996.

\bibitem{qian}
Z.~Qian.
\newblock Estimates for weighted volumes and applications.
\newblock {\em Quart. J. Math. Oxford Ser. (2)}, 48(190):235--242, 1997.

\bibitem{WW}
G.~Wei and W.~Wylie.
\newblock Comparison geometry for the {B}akry-{E}mery {R}icci tensor.
\newblock {\em J. Differential Geom.}, 83(2):377--405, 2009.

\bibitem{MR2567278}
J.-Y. Wu.
\newblock Upper bounds on the first eigenvalue for a diffusion operator via
  {B}akry-\'{E}mery {R}icci curvature.
\newblock {\em J. Math. Anal. Appl.}, 361(1):10--18, 2010.

\bibitem{W}
W.~Wylie.
\newblock Sectional curvature for {R}iemannian manifolds with density.
\newblock {\em Geom. Dedicata}, 178:151--169, 2015.

\bibitem{Zhu}
S.~Zhu.
\newblock The comparison geometry of {R}icci curvature.
\newblock In {\em Comparison geometry ({B}erkeley, {CA}, 1993--94)}, volume~30
  of {\em Math. Sci. Res. Inst. Publ.}, pages 221--262. Cambridge Univ. Press,
  Cambridge, 1997.

\end{thebibliography}

\def\cprime{$'$} \def\cprime{$'$} \def\cprime{$'$} \def\cprime{$'$}

 \end{document}